\documentclass[11pt
]{article}
\usepackage{textcomp,upgreek}
\usepackage[mathcal]{euscript}
\usepackage{tipa,universal} 
\usepackage{rotating}
\usepackage{pstricks}
\usepackage{amssymb,amsmath,theorem,latexsym}
\usepackage{graphicx,rotating,graphics,epsfig,epic,pifont,color}
\onecolumn \textwidth = 6.63in\textheight = 9.5in\oddsidemargin = 0in \evensidemargin = 0in\topmargin = -20mm
\DeclareMathAlphabet{\mathpzc}{OT1}{pzc}{m}{it}
\newrgbcolor{mygray}{.9 .9 .9}
\newrgbcolor{mygray2}{.4 .4 .4}
\def\jala{\hspace{-1mm}}

    \def\cf{{\em cf.\ }}
    \def\g{\mbox{\rm g}}
    \def\sat{\mbox{sat}}

    \def\tmax{t^{\mbox{\scriptsize max}}}
    \def\rref#1{(\ref{#1})}
    \def\ep{\varepsilon}
    \def\qd{q_d}
    \def\dqd{\dot q_d}
    \def\ddqd{\ddot q_d}
    \def\dtq{\dot {\tilde q}}
    \def\tq{{\tilde q}}
    \def\mR{\mathbb{R}}
    \def\mRp{\mathbb{R}_{\geq 0}}
    \def\ie{{\it i.e.}}
    \def\dty{\displaystyle}
    \def\opchi{\shrinkthis[0.65]{.92}{\mathcal X}}
    \def\opxi{\shrinkthis[0.85]{1}{\xi}}
    \def\topxii{\shrinkthis[0.95]{1}{\tilde\xi}_i}
    \def\topxi#1{\shrinkthis[0.95]{1}{\tilde\xi}_{#1}}

    \def\dopxi#1{\shrinkthis[0.95]{1}{\dot\xi}_{#1}}
    \def\dtopxi#1{\shrinkthis[0.95]{1}{\dot{\tilde\xi}}_{#1}}
    
    \def\kronecker{\mbox{\small $\stackrel{\otimes}{}$}}
    \def\dopchi{{\dot\opchi}}
       \newcommand\dgdots[1]{%
      \begin{rotate}{#1}
       $\vdots$
      \end{rotate}
    }

    \makeatother
    \catcode`\@=11
    \def\downparenfill{$\m@th\braceld\leaders\vrule\hfill\bracerd$}
    \def\overparen#1{\mathop{\vbox{\ialign{##\crcr\crcr
    \noalign{\kern0.4ex}
    \downparenfill\crcr\noalign{\kern0.4ex\nointerlineskip}
    $\hfil\displaystyle{#1}\hfil$\crcr}}}\limits}
    \catcode`\@=12\makeatletter

    \newcommand\vecttwo[2]{
      \begin{bmatrix}
        \norm{#1}\\[1mm] \norm{#2}
      \end{bmatrix}
    }
    \newcommand{\mattwo}[3]{
      \begin{bmatrix}
        #1 & #2 \\[1mm] * & #3
      \end{bmatrix}
    }

    \newcommand{\norm}[1]{\left|#1\right|}
    \newcommand{\sfrac}[2]{\mbox{\small $\dty\frac{#1}{#2}$} }

    \newcommand\shrinkthis[3][1.08]{ \scalebox{#2}[#1]{\mbox{$#3$}} }
    \newcommand{\tonio}[1]{\par {\color{red}{\bf T: #1}\par}}
    \renewcommand{\tonio}[1]{}
    \newtheorem{definition}{Definition}
    \newtheorem{theorem}{Theorem}
    \newtheorem{assumption}{Assumption}
    \newtheorem{proposition}{Proposition}
    \newtheorem{corollary}{Corollary}
    \newtheorem{lemma}{Lemma}
{\theorembodyfont{\rm} 
    \newtheorem{remark}{Remark}
    \newtheorem{example}{Example}}
    \def\qed{\mbox{}\hfill{$\blacksquare$}}
    \newenvironment{proof}{\noindent {\bf Proof. }}{\qed}

\title{ \Large\bf \vspace{5mm}%
{Observer-less Output Feedback Global Tracking Control\\ of Lossless Lagrangian Systems}}
\author{ {\Large
Antonio Lor\'{\i}a} \ \\[2mm]
CNRS, France\\[-1mm]
{\small\tt loria@lss.supelec.fr}} \vskip 10pt
\date{}
\begin{document}

\maketitle

\parskip=2pt plus 1pt minus 1pt
\global\setlength\theorempreskipamount{2pt} 
\global\setlength\theorempostskipamount{2pt} 

\begin{abstract}
We obviate the use of observers for the purpose of output feedback tracking control of Lagrangian systems and solve some long-standing yet well-documented open problems. As often implemented in control practice, we replace unavailable derivatives with approximate differentiation. Our contribution consists in establishing  {\em uniform global} asymptotic stability in closed-loop, for Lagrangian systems without dissipative forces (friction) using only position feedback. Firstly, for fully-actuated relative-degree-two systems, the controller is reminiscent of passivity-based controllers for robot manipulators and consists in a {\em linear} dynamic system together with a globally-Lipschitz  control law. Establishing a {\em global} uniform result, all the more with such a simple controller, is particularly valuable relatively to the literature of output-feedback control of systems with non-globally-Lipschitz nonlinearities in the unmeasured variables. This first contribution solves a long-standing open problem and, as a matter of fact, recasted in a general context this result is at the edge of what is achievable --see \cite{MAZPRADAY}. Then, we show that our control approach may be applied to a more general problem, that of tracking control of Lagrangian systems augmented by a chain of integrators (relative degree $m+2$ systems). As a corollary, we solve the global-tracking control problem for {\em flexible-joint} robots but also for systems coupled with output-feedback linearizable actuator dynamics. Finally, we discuss remaining open problems of fairly general interest in the realm of  analysis and design of robust nonlinear systems.
\end{abstract}

\section{Introduction}\label{sec:intro}
\tonio{Modify the introduction to speak of higher-relative degree systems.}

We study Euler-Lagrange systems defined by the equation
\begin{eqnarray}\label{10}
D(q) \ddot q + C(q,\dot q)\dot q + g(q) =u
\end{eqnarray}
where $q\in\mR^n$ denotes the generalized positions, $\dot q$ denotes the generalized velocities, $D:\mR^n\to\mR^{n\times n}$ corresponds to the inertia matrix function, $C:\mR^n\times \mR^n\to\mR^{n\times n} $ corresponds to the Coriolis and centrifugal forces matrix function, $g:\mR^n\to \mR^n$ represents the vector of forces which are derived from the potential energy function $U: \mR^n\to \mR$ \ie, $g(q):= \frac{\partial U}{\partial q}(q)$ and $u\in\mR^n$ is the vector of control inputs. All functions are assumed to be once-continuously differentiable in their arguments.

We revisit the problem of output-feedback tracking control, which consists in designing a dynamic  controller with output $u$ that  makes use of $q$ as the only plant measurement and ensures that, given a smooth bounded trajectory $t\mapsto \qd$, the generalized coordinates satisfy 
\[
\lim_{t\to \infty} q(t) \to \qd(t),  \quad \lim_{t\to \infty} \dot q(t) \to \dqd(t).
\]
More precisely (and of much higher difficulty) the problem that we address consists in establishing uniform global asymptotic stability of the origin of the closed-loop system. We put special emphasis on the qualifier global which implies that the property must hold for all initial states of the closed-loop system, including the tracking errors in $\mR^{2n}$ as well as the controller's states. Not to be confused with inappropriate terminologies such as ``global on the set $\mathcal X \subset \mR^n$''  or the weaker property ``global in the plant's variables and semi-global in the controller's'', established in other related articles.

In the last 25 years or so there have been numerous attempts to solve the problem mentioned above, as a particular paradigm of dynamic output feedback control of nonlinear systems. See for instance  \cite{Marino1991115}, as well as other works by the same authors, on output feedback linearization. In a similar train of thought we find methods that rely on the ability to perform a coordinate transformation of system \rref{10} into models that are linear in the unmeasured velocities. See for instance the work of G. Besan\c{c}on --\cite{GILDASPHD,BES00} and subsequent references. However, it has been long recognized that such transformations are inapplicable to many physical systems; even to simple two-degrees-of-freedom planar robots with revolute joints --see \cite{SPOFACT}. 

Other works focus on robot tracking control. For instance, in the classic paper \cite{BURNAN} the author presents a proof of uniform asymptotic stability using the controller for set-point regulation --indeed, the latter was independently published in \cite{KEL1} and \cite{BERNIJ}. In \cite{BURNAN} the author invokes Tychonov's theorem to show uniform global asymptotic stability provided that the unique pole of the dirty-derivatives filter used in \cite{KEL1} is placed at $-\infty$ that is, the result actually establishes semi-global asymptotic stability. The same property is achieved via Lyapunov's direct method in \cite{BER,al:LORORT}. Relying on the practically reasonable but theoretically restrictive assumption that the system possesses natural viscous friction, the authors of \cite{NUNHSU} established global asymptotic stability. That is, the model considered in this reference is 
\begin{equation}\label{confriccion}
  D(q) \ddot q + C(q,\dot q)\dot q + F\dot q + g(q) =u
\end{equation}
where $F$ is symmetric positive definite. However, under these conditions it is direct to extend the stability property from semi-global to global, for a number of results in the literature.

To the best of the author's knowledge, it has not been established either that uniform global asymptotic stability via output feedback is out of reach  for the system \rref{10}. What is more, it rather seems that the absence of proof or disproof has simply eluded the efforts of many researchers (including this author) throughout the years and is not due to a $structural$ impediment.  This is investigated in the seminal article  \cite{MAZPRADAY}, where the concept of unboundedness observability is introduced. Roughly, from the main results in \cite{MAZPRADAY} it may be concluded that the system
\[
d\ddot q + c\dot q^2 = u, \quad q,\, u\in\mR
\]
cannot be stabilized globally via dynamic output feedback with output $q$. The obstacle is that the system does not possess the unboundedness observability property from $q$ that is, the solution $[\dot q(t), q(t)]$ may escape to infinity even for bounded values of $q(t)$. Notice that this is not the case of Lagrangian systems, which possess the structural property of skew-symmetry of the matrix $\dot{\overparen{D(q)}} - 2C(q,\dot q)$. Indeed, uniform global asymptotic stability of systems
\[
d(q)\ddot q + c(q)\dot q^2 + g(q) = u 
\]
is established in \cite{al:onedof} provided that $\partial d(q)/\partial q = 2c(q)$. As a matter of fact, to our knowledge this is the only article that presents a dynamic output-feedback controller for Euler-Lagrange systems together with a strict Lyapunov function albeit for {\em one}-degrees-of-freedom systems. The extension of \cite{al:onedof} to the case of $n$-degree-of-freedom systems has not been obtained: {attempts include \cite{new-pinDAW,PETACC06} however, the controller from \cite{new-pinDAW} is guaranteed (in the non-adaptive case, only) to achieve uniform asymptotic stability for any system's initial conditions provided that the controller's trajectories lay in a forward-invariant set. Moreover, the result in \cite{new-pinDAW} relies on the assumption that the model includes viscous friction (of known magnitude in the non-adaptive case) --see Eq. \rref{confriccion}, and that the forces derived from potential energy are bounded. }The controller of \cite{PETACC06} is in implicit form hence it is not implementable without velocity measurements.

Roughly speaking, there are two types of results addressing this problem; those based on Lyapunov's direct method and those which intend to exploit structural properties. In the first case, stability is global only with respect to part of the states --as in \cite{new-pinDAW}, or is  semiglobal --as in \cite{BURNAN,al:LORORT}, among others. In the second case, the structural assumptions needed to perform convenient changes of coordinates do not hold for EL systems --{\em cf.} \cite{SPOFACT}. An exception is the recent PhD thesis \cite{LUPE} --see also \cite{LUPEACC13}, where the author presents a global result\footnote{The thesis \cite{LUPE} and the proceedings article \cite{LUPEACC13} were presented during the preparation of this paper. See also the independent simultaneous article \cite{ACC13-globtrack} which constitutes a preliminary version of this paper.} for {\em Hamiltonian} systems that relies on a clever but intricate observer-design due to \cite{Astolfi2010182} and a change of coordinates that involves the computation of the square root of\footnote{Although difficult in general, the author of \cite{LUPE} stresses that this may be computed numerically for each fixed $q$.} $D(q)^{-1}$ --see \cite{al:GER99,al:GLOBTRACKTAC} in the same spirit.

The rest of the paper is organized as follows. For the sake of clarity, in Section \ref{sec:defs}  we recall basic stability definitions. In Section \ref{sec:globtrack} we present our first result which solves the open problem described above, for systems with relative degree 2 ({\em e.g.}, fully-actuated Lagrangian systems). In Section \ref{sec:extension} we extend our results to the case of higher relative degree. In Section \ref{sec:disc} we provide some comments and discuss open problems, before concluding with some remarks in Section \ref{sec:concl}.

\section{Preliminaries}
\label{sec:defs}
\noindent {\bf Notation}. Recall that a continuous function $\alpha:\mRp\to\mRp$ is of class $\mathcal K$ if it is strictly increasing and $\alpha(0)=0$, a continuous function $\sigma:\mRp\to\mRp$ is of class $\mathcal L$ if it is strictly decreasing and $\sigma(s)\to 0$ as $s\to \infty$; a continuous  function $\beta:\mRp\times\mRp\to\mRp$ is of class $\mathcal K\mathcal L$ if $\beta(r,\cdot)\in \mathcal L$ and $\beta(\cdot,s)\in \mathcal K$; a continuous function $\alpha:\mRp\to\mRp$ is of class $\mathcal K_\infty $ if $\alpha\in\mathcal K$ and $\alpha(s)\to \infty$ as $s\to\infty$. We denote by  $\norm{\,\cdot\,}$, the Euclidean norm of vectors (or any other compatible norm) and the induced norm of matrices.

To remove all possible ambiguity we recall a few definitions of stability from \cite{HAH} and some statements which are either known or are re-stated in an original manner, for the purposes of this article. Consider the dynamic system
\begin{equation}
\label{one}
\dot x = f(t,x), \quad x\in\mR^n,\, t \in \mRp
\end{equation}
where $f$ satisfies the conditions for existence and uniqueness of solutions and $f(t,0)\equiv 0$. We denote by $x(t,t_\circ,x_\circ)$ or, when the context is clear, by $x(t)$ the solutions of \rref{one} with initial times $t_\circ\in\mRp$ and initial states $x_\circ\in\mR$ that is, we have  $x(t_\circ,t_\circ,x_\circ) = x_\circ$.

\begin{definition}[Uniform global boundedness]
  \label{def:ugb}
The solutions of \rref{one} are said to be uniformly globally bounded if there exist $\gamma\in {\cal K}_\infty$ and $c>0$ such that, for all $(t_{\circ},x_\circ) \in \mRp \times \mR^{n}$, each solution $x(\cdot,t_{\circ},x_{\circ})$ satisfies 
\begin{equation}
\label{eq:ugb}
  \norm{x(t,t_{\circ},x_{\circ})} \leq \gamma(\norm{x_{\circ}}) + c  \qquad \forall \; t \geq t_{\circ}.
\end{equation}
\end{definition}
Note that for any $r$ there exists $R$ independent of $t_{\circ}$ such that $\norm{x_{\circ}}\leq r$ implies that $\norm{x(t,t_{\circ},x_{\circ})} \leq R$. This property is commonly established via auxiliary functions.
\begin{theorem}
\label{thm:ugb}Let $V:\mRp\times\mR^n\to \mRp$ be continuously differentiable; let $\alpha_1$, $\alpha_2$ be functions of class $\mathcal K_\infty$ and let $a$, $c\in\mR$ be such that $c>0$ and 
\[
\alpha_1(\norm{x}) \geq V(t,x) \geq \alpha_2(\norm{x}) + a \qquad \forall\ (t,x) \in\mRp\times\mR^n
\]
\[
  \dot V(t,x) := \frac{\partial V}{\partial t} + \frac{\partial V}{\partial x}f(t,x) \leq 0 \qquad \forall\ t\in\mRp,\ x :  \norm{x}\geq c 
\]
Then, the solutions of \rref{one} are uniformly globally bounded.
\end{theorem}
\begin{proof}
For all $t$ such that $\norm{x(t)}\geq c$ we have $\dot V(t,x(t))\leq 0$ that is, $\alpha_1(\norm{x_\circ}) \geq V(t_\circ,x(t_\circ))\geq \alpha_2(\norm{x(t)})+a$. Hence, for all $t\in \mRp$ we have $\norm{x(t)} \leq \alpha_2^{-1}\circ\Big(\alpha_1(\norm{x_\circ})+\norm{a}\Big) + c$. Since $\alpha_2^{-1}\in\mathcal K_\infty$ there exist $\gamma_1$, $\gamma_2\in \mathcal K_\infty$ such that $\norm{x(t)} \leq \gamma_1(\norm{x_\circ}) + \gamma_2(\norm{a})  + c$
\end{proof}

Although unusual in ``modern'' literature, the following fundamental definition may be found for instance, in \cite{HAH}.
\begin{definition}[Uniform global stability]
  \label{def:stability}
  The origin of system \rref{one} is said to be uniformly  globally stable  if there exists $\gamma\in {\cal K}_\infty$ such that for each $(t_{\circ},x_\circ) \in \mRp \times \mR^{n}$, each solution $x(\cdot,t_{\circ},x_{\circ})$ satisfies 
\begin{equation}
\label{us}
  \norm{x(t,t_{\circ},x_{\circ})} \leq \gamma(\norm{x_{\circ}})  \qquad \forall \; t \geq t_{\circ}.
\end{equation}
\end{definition}
Note that uniform global stability tantamounts to uniform stability plus uniform global boundedness.

\begin{theorem}
  Let the conditions of Theorem \ref{thm:ugb} hold for $a=c=0$. Then, the origin of \rref{one} is uniformly  globally stable. If the conditions hold only in an open neighborhood of the origin with $\alpha_1$, $\alpha_2\in \mathcal K$, the latter is uniformly stable.
\end{theorem}
\begin{proof}
The statement on globality: following the proof of Theorem \ref{thm:ugb} we have $\norm{x(t)} \leq \alpha_2^{-1}\circ\Big(\alpha_1(\norm{x_\circ})\Big)$. The ``local statement'', which  is due to Persidsk\u{\i}i --\cite{PERSIDSKII33}, follows by restricting $x_\circ$ to a neighborhood such that $\alpha_1(\norm{x_\circ})$ belongs to the domain of $\alpha_2^{-1}$ and appears in numerous textbooks.
\end{proof}
\begin{definition}[Uniform global attractivity]
The origin of system (\ref{one}) is said to be  uniformly globally attractive if  for each  $r,\,\sigma>0$  there exists $T>0$ such that
\begin{equation}
\label{UA}
\norm{x_{\circ}}\leq r  \, \Longrightarrow \, \norm{x(t,t_{\circ},x_{\circ})}
\leq \sigma \qquad  \forall\, t \geq t_{\circ} + T\,
  \ .
\end{equation}
\end{definition}
\begin{definition}[Uniform Global Asymptotic Stability]\label{def:ugas}
The origin  of system \rref{one} is said to be uniformly globally asymptotically stable if it is 
\begin{itemize}\itemsep=-1pt
\item uniformly stable;
\item the solutions are uniformly globally bounded;
\item the origin is uniformly globally attractive.
\end{itemize}
\end{definition}
It is important to emphasize that only all three conditions in Definition \ref{def:ugas} {\em together} imply the existence of a class $\mathcal K\mathcal L$ function $\beta$ such that the solutions of \rref{one} satisfy 
\[
\norm{x(t)}\leq \beta(\norm{x_\circ},t-t_\circ)\qquad \forall \, t \geq t_\circ\geq 0.
\]
The latter leads to the construction of converse Lyapunov functions uniformly monotone and, in turn, implies robustness with respect to external perturbations. Such bound cannot be obtained if any of the three properties in Definition \ref{def:ugas} fails to hold. In particular, uniform global asymptotic stability may not be concluded either from uniform stability plus uniform global attractivity alone --see \cite{TEEZAC}; whence the importance of uniform global boundedness in nonlinear time-varying systems.

The following theorem, which corresponds to {\rm \cite[Lemma 2]{al:TACDELTAPE}}, establishes uniform global asymptotic stability without a strict Lyapunov function, in the spirit of integral criteria such as Barbal\u{a}t's lemma.
\begin{theorem}
\label{thm:uga}
Let $F:\mRp \times \mR^{n} \rightarrow \mR^{n}$ be continuous such that $F(\cdot,0)\equiv 0$. Let the equilibrium  $\{x=0\}$ of $\dot{x} = F(t,x)$ be uniformly globally stable. If moreover, there exists a continuous positive definite function $\gamma: \mR^{n} \to \mRp$ and for each pair of positive numbers $(r,\tilde\upsilon)$, there exists $\beta_{r\tilde\upsilon} > 0$ such that, for all $(t_\circ,x_\circ)\in \mRp\times B_r$, all solutions $x(\cdot,t_\circ,x_\circ)$ satisfy 
\begin{eqnarray}
 \label{eq:int-lem}
\int_{t_\circ}^{t}\Big[\gamma\big(x(\tau,t_\circ,x_\circ)\big)-{\tilde\upsilon}\Big]d\tau \leq \beta_{{r\tilde\upsilon}}
\end{eqnarray}
for all $t \geq t_{\circ}$ then, the origin is  uniformly globally asymptotically stable.
\end{theorem}
\begin{remark}
The origin being assumed to be uniformly globally stable, the condition imposed by Ineq. \rref{eq:int-lem} guarantees uniform global attractivity. Roughly speaking, \rref{eq:int-lem} implies that the trajectories are integrable, modulo $\gamma(\cdot)$,  out of {\em any} ball of radius depending on $\tilde \upsilon$; this implies that $\gamma(x(t))$ converges to the interior of such ball in finite time. Since $\tilde\upsilon$ is arbitrary, we conclude that \rref{UA} holds. See \cite{al:TACDELTAPE} for rigorous proof.
\end{remark}

\section{Relative-degree-2 systems}
\label{sec:globtrack}

\subsection{Problem statement and solution}
The following assumptions are fairly standard in the literature of robot control but are also satisfied by a number of Euler-Lagrange systems such as electrical and electro-mechanical --see  \cite{al:PBCELS}, as well as some marine systems --see  \cite{THORB}. In particular, these hypotheses hold for robot manipulators composed of revolute joints only or prismatic joints only.
\begin{assumption}
\label{ass1}\
\begin{enumerate}
\item There exist positive real numbers $d_m$ and $d_M$ such that 
\begin{equation*}
d_m \leq \norm{D(q)} \leq d_M, \qquad \forall q\in \mR^{n};
\end{equation*}  
\item  there exists  $k_c > 0$ such that  
\begin{subequations}
\begin{eqnarray*}
\norm{C(x,y)} &\leq & k_c\norm{y}\hspace{11.5mm} \forall\, x,\, y \in \mR^n, \\
C(x,y)z &=& C(x,z)y \qquad \forall\, x, y, z \in \mR^n ;
\end{eqnarray*}  
\end{subequations}
\item the matrix $\dot{\overparen{D(q)}} - 2C(q,\dot q)$ is skew symmetric.
\end{enumerate}
\end{assumption}

\begin{definition}[global output-feedback tracking control]\label{def:pblm}
Consider the EL system \rref{10} under Assumption \ref{ass1}. Suppose that only position measurements are available, and that the properties enumerated in Assumption \ref{ass1} hold. Furthermore, assume that the reference trajectory  $t\mapsto q_d$ is of class ${\cal C}^2$ and that there exists $k_\delta >0$ such that
\begin{equation}
\label{bndonqd}
\max\,\left\{ \sup_{t\geq 0}\norm{\qd(t)},\, \sup_{t\geq 0}\norm{\dqd(t)},\, \sup_{t\geq 0}\norm{\ddqd(t)}\right\} \leq k_\delta\, .
\end{equation}
Under these conditions, find a dynamic output-feedback controller 
\begin{subequations}\label{bla}
  \begin{eqnarray}
\label{bla:a} \dot q_c &=& f(t,q_c,q) \\
\label{bla:b} u &= & u(t,q_c,q)  
\end{eqnarray}
\end{subequations}
such that the closed-loop system
\begin{eqnarray}
\label{five}
&\displaystyle  D(q) \ddot q + C(q,\dot q)\dot q + g(q) = u(t,q_c,q) & \\
\nonumber 
&\dot q_c = f(t,q_c,q) &
\end{eqnarray}
has a unique equilibrium at 
\begin{eqnarray*}
&\displaystyle  (\tilde q,\dtq,q_c-q_c^*) = (0,0,0), &\\
&\displaystyle \tilde q : = q-q_d(t), \quad \dot {\tilde q} := \dot q - \dot q_d(t)&
\end{eqnarray*}
where  $q_c^*$ is a solution to \rref{bla} with $q\equiv q_d$ and this equilibrium is {\em uniformly globally asymptotically stable}. 
\end{definition}

\noindent Our first theorem improves the main result in \cite{ACC13-globtrack} which solves the long-standing open problem defined above. 

\begin{theorem}\label{thm:main1}
  Consider the system  \rref{10} under Assumption \ref{ass1}. Let $a$, $b$, $k_p$ and $k_d$ be positive constants satisfying 
      \begin{eqnarray}
\label{11}   \frac{k_d}{2}\frac{b}{a} > k_ck_\delta
  \end{eqnarray}
and consider the dynamic position-feedback controller
\begin{subequations}\label{90}
\begin{eqnarray}
\label{90a}\dot q_c &=& -a(q_c + b\tilde q) \\
\label{90b}\vartheta & = & q_c + b\tilde q\\
\label{90c}u & = & -k_p \tilde q - k_d \vartheta + D(q)\ddot q_d + C(q,\dot q_d)\dot q_d + g(q). \qquad \
\end{eqnarray}
\end{subequations}
Then, the origin $\{z=0\}$ with $z:= [ \tilde q^\top\, \dot {\tilde q}^\top \, q_c^\top ]^\top$ is uniformly globally asymptotically stable.
\end{theorem}
\begin{remark}
In the statement of Theorem \ref{thm:main1} we use scalar gains $a$, $b$, $k_p$ and $k_d$ purely for clarity of exposition. The result holds if the gains are defined as diagonal positive matrices, and replacing condition \rref{11} with
\[
\frac{k_{d_m}}{2}\frac{b_m}{a_M} > k_ck_\delta
\]
where $(\cdot)_m$ and $(\cdot)_M$ denote respectively, to the smallest and largest elements in the diagonal of $(\cdot)$.
\end{remark}

The controller \rref{90} is based on its ``set-point controller'' counter-part, first published in \cite{KEL1} and subsequently used by many other authors. In \cite{ACC13-globtrack} it is proved that {\em there exists} a finite value of $a$ such that the origin is uniformly globally asymptotically stable. This is also claimed in \cite{BURNAN} in which the author relies on singular perturbation theory to actually establish that the domain of attraction may be extended to $\mR^{3n}$ provided that $a\to\infty$ (hence, establishing semiglobal rather than global asymptotic stability). Reminiscent of that in \cite{BURNAN}, the controller \rref{90} corresponds {\em verbatim} to that from \cite{al:LORORT} where {\em semi}~\!global uniform asymptotic stability is also established under much more stringent conditions --than \rref{11}--  on the control gains. 

Even though the controller \rref{90} is not original, the extent of Theorem \ref{thm:main1} may hardly be overestimated:
\begin{itemize}
\item it establishes uniform global asymptotic stability of the origin of the closed-loop, thereby solving a problem open for about 25 years;
\item the control input $u(t,\tq,\dtq,\vartheta)$ is defined by a globally-Lipschitz map uniformly bounded in $t$ and $\tilde q$ and  the controller dynamics is linear;
\item the sole condition on the control gains is \rref{11} which is in great contrast to the conditions in \cite{al:LORORT} and other articles where, at best, only uniform semiglobal asymptotic stability is established;
\item it establishes that for Lagrangian systems, the fact of introducing damping through a low-pass filter, does not alter the global property of its state-feedback counter-part. 
\end{itemize}

From a practical viewpoint the controller \rref{90} is fairly easy to implement and the sole tuning rule is both simple and practically meaningful. As a matter of fact, \rref{90} is reminiscent of the most elementary control strategies and employs a widely-used {\em ad hoc} alternative to a differentiator; Eqs. \rref{90a}, \rref{90b} correspond to the state-space representation of the so-called dirty-derivatives filter 
\begin{equation}\label{filter}
  \vartheta = \frac{b}{a+s} \dtq
\end{equation}
whose output is commonly used in control practice to replace the unavailable velocities $\dtq$. The transfer function \rref{filter} corresponds to that of a low-pass filter with DC gain equal to $b/a$ and the cutting frequency determined by $a$. 

Note that the {\em  only} condition, inequality \rref{11}, imposes a very natural constraint on the damping gain since it is assumed that the system is naturally lossless (it has no natural dissipative forces such as friction). Such damping may be introduced {\em directly} through the control law, by choosing $k_d$ large relative to the Lipschitz constant on the Coriolis forces and to the upper-bound on the reference trajectories and their derivatives. In such case, the positive values of $a$ and $b$ only affect the performance but do not condition the stability properties. Otherwise, if for instance the control inputs are constrained, one may choose to keep the gains $k_p$ and $k_d$ relatively small and introduce the necessary damping {\em indirectly} through the filter gain, by choosing the ratio $b/a$ large enough to satisfy \rref{11}. Certainly, the farthest the pole of the low-pass filter \rref{filter} is placed away from the origin, the better performance is to be expected. Indeed, typically one chooses $b=a\gg 1$ for unitary gain and to place the cutting frequency high enough to attenuate little the higher harmonics of $\dtq(t)$.  For a few simulation tests see \cite{al:LORORT} and \cite{al:thesis}.

From a pure perspective of dynamical systems Theorem \ref{thm:main1} establishes for \rref{10}, the strongest property desirable for a nonlinear time-varying system: uniform global asymptotic stability\footnote{Note that exponential and finite-time stability are particular cases of the latter.}. Achieving such property for nonlinear time-varying systems with non-globally Lipschitz nonlinearities of the unmeasured variables is at the edge of the achievable via dynamic output feedback --see \cite{MAZPRADAY} and Section \ref{sec:disc}. Moreover, globality is achieved via a {\em linear} autonomous controller with a globally-Lipschitz control input $u$.

\subsection{Proof of Theorem \ref{thm:main1}} 

The closed-loop equation is obtained by replacing $u$ from \rref{90c} in \rref{10} and adding $-C(q,\dqd)\dot q + C(q,\dot q)\dot q_d = 0$ to the right-hand side of \rref{90c} hence, 
\begin{equation}\label{100:b}
  D(q)\ddot {\tilde q} + [\,C(q,\dot q) + C(q,\dot q_d)]\dot {\tilde q} + k_p \tilde q + k_d \vartheta  = 0
\end{equation}
and, for the purpose of analysis, we differentiate \rref{90b} and use \rref{90a} to obtain 
\begin{equation}
  \label{100:a}  \dot \vartheta = -a \vartheta + b\dot{\tilde q}.
\end{equation}
The point $\{x = 0\}$ where  $x: = [\tq^\top \, \dtq^\top \, \vartheta^\top]^\top$ is an equilibrium of \rref{100:b}, \rref{100:a} and is unique. Then, a direct computation shows that $\{z=0\}$ is a unique equilibrium of the closed-loop equations  \rref{90a}, \rref{100:b}. Also, $\{x=0\}$ is uniformly globally asymptotically stable for  \rref{100:b}, \rref{100:a} if and only if so is $\{z=0\}$ for the closed-loop equations \rref{90a}, \rref{100:b}. We proceed to analyze the stability of $\{x = 0\}$.

The analysis of \rref{100:b}, \rref{100:a} is divided in four ordered steps which establish uniform forward completeness --see Lemma \ref{prop:zero} below, and the three properties listed in Definition \ref{def:ugas}. Lemma \ref{prop:one}, farther below, establishes uniform global boundedness. For completeness, uniform stability is contained in the subsequent Lemma \ref{prop:two} however, this property  follows trivially via Lyapunov's first method and it is also implicitly contained in the proof of the main result in \cite{al:LORORT}. Finally, the main statement of Lemma \ref{prop:two} establishes uniform global attractivity.

\begin{lemma}\label{prop:zero}
Under the conditions of Theorem \ref{thm:main1} the closed-loop system \rref{10}, \rref{90} is uniformly forward complete; moreover, there exist $c_1$, $c_2>0$ such that 
  \begin{equation}
    \label{eq:110}
\norm{x(t)} \leq c_1 \norm{x(t_\circ)} \mbox{e}^{c_2(t-t_\circ)}.
  \end{equation}
\end{lemma}
\begin{proof}
Consider the Lyapunov function candidate $V:\mRp\times\mR^{3n}\mapsto\mRp$ defined as 
\begin{equation}
  \label{120}
V_1(t,\tq,\dtq,\vartheta) = \frac{1}{2}\left(
\dtq^\top D(\tq+\qd(t)) \dtq + k_p \norm{\tq}^2 + \frac{k_d}{b} \norm{\vartheta}^2 \right) 
\end{equation}
which, under Assumption \ref{ass1}, satisfies $\alpha_1\norm{x}^2 \geq V_1(t,\tq,\dtq,\vartheta)\geq \alpha_2\norm{x}^2$ with 
\[
\alpha_1:= \sfrac{1}{2}\max\{ d_M,\, k_p,\, \sfrac{k_d}{b}\} \quad \alpha_2:= \sfrac{1}{2}\min\{ d_m,\, k_p,\, \sfrac{k_d}{b}\}. 
\]
Furthermore, using $\dot{\overparen{D(q)}}=C(q,\dot q)+C(q,\dot q)^\top$, we see that the total derivative of $V_1$ along the closed-loop trajectories of \rref{100:b}, \rref{100:a}, satisfies
\begin{eqnarray}
\nonumber \dot V_1 &=& -\frac{k_da}{b}\norm{\vartheta}^2 + \dtq^\top C(q,\dot q_d) \dtq\\
\label{140:b} &\leq &  -\frac{k_da}{b}\norm{\vartheta}^2 + k_ck_\delta\norm{\dtq}^2.
\end{eqnarray}
Therefore, $\dot V_1 \leq k_ck_\delta\norm{x}^2$ and, defining $v_1(t):= V_1(t,\tq(t),\dtq(t),\vartheta(t))$, we obtain $\dot v_1(t) \leq (k_ck_\delta/\alpha_2) v_1(t)$. The statement follows integrating the latter, defining
\[
c_1:= \sqrt{\sfrac{\alpha_1}{\alpha_2}}, \quad c_2 := \sfrac{k_ck_\delta}{2\alpha_2}
\]
 and invoking the comparison principle. 
\end{proof}

Next, we introduce a statement which follows along the lines of \cite[Theorem 2]{Kalsi} and is used in the proof of uniform global boundedness --see Lemma \ref{prop:one}.
\begin{lemma}\label{lem:kalsi}
Consider the differential equation 
\begin{equation*}
 \dot \vartheta = -a\vartheta + \nu(t), \quad \nu:\mR_{[t_\circ,\tmax)}\to \mR^n, \ t_\circ \in \mRp, \ \tmax\in [t_\circ,\infty]
\end{equation*}
where $\nu$ is uniformly continuous and bounded. Let $\tau \in (t_\circ,\tmax)$ and $\epsilon := 1/a$; if $\nu$ is uniformly continuous then 
\begin{equation}
  \label{92} 
\displaystyle \lim_{\epsilon \to 0} \dot \vartheta(t) = 0 
\end{equation}
 uniformly for all $t\in[\tau,\tmax)$.
\end{lemma}
\begin{lemma}\label{prop:one}
Under the conditions of Theorem \ref{thm:main1}, the closed-loop trajectories of the system \rref{10}, \rref{90} are uniformly globally bounded.
\end{lemma}
\begin{proof}
We analyze the solutions to \rref{100:b}, \rref{100:a} with initial conditions $t_\circ\geq 0$ and $x(t_\circ)= x_\circ\in B_r$ where $r>0$ is arbitrarily fixed. 

We apply Lemma \ref{lem:kalsi} to Equation \rref{100:a} with $\nu(t) = b\dtq(t)$. To that end, we first observe that \rref{eq:110} implies that for any $r>0$ and  $\tmax>t_\circ$ there exists $M$ such that $\max\{\norm{x(t)}, \norm{\dot x(t)}\} \leq M(\tmax,r)$ for all\footnote{The inequality for $ \norm{\dot x(t)}$ follows from forward completeness and Assumption \ref{ass1}. } $t\in [t_\circ,\tmax)$ and all $x_\circ\in B_r$. Therefore, $\nu$ is bounded and uniformly continuous on  $[t_\circ,\tmax)$. Now, the expression \rref{92} implies that for any $\Delta>0$ there exists $a^*$ such that 
\begin{equation}
  \label{320}
a\geq a^*(\Delta) \Rightarrow \norm{\dot \vartheta(t)} \leq \Delta, \quad \forall\, t\in [\tau,\tmax).
\end{equation}
We emphasize that $a^*$ only depends on $\Delta$ and the latter is independent of $M$ (hence, independent of $r$) since the rate of convergence of $\dot\vartheta$ in \rref{92} is independent of the bound on $\nu(t)$ and is uniform in $t$ --see \cite{Kalsi}. Also, in view of \rref{92} it follows that $\Delta\mapsto a^*$ is strictly decreasing and $a^*\to\infty$ as $\Delta\to 0^+$ therefore, let $\Delta^*$ be large enough so that $\Delta=\Delta^*>0$ generate,  via \rref{320}, $a^*>0$ such that\footnote{This is always possible since otherwise, we can redefine $\Delta^*_{\mbox{\tiny new}} > a^* > \Delta^*$ and \rref{320} continues to hold for $\Delta= \Delta^*_{\mbox{\tiny new}}$}  $\Delta^*/a^* \geq 1$. From \rref{100:a} we see that for any $t\in [\tau,\tmax)$ and $a\geq a^*$,
  \begin{equation}\label{92.5}
    b\norm{\dtq(t)} \leq \Delta^* + a\norm{\vartheta(t)}.
  \end{equation}
Recalling that $v_1(t)= V_1(t,\tq(t),\dtq(t),\vartheta(t))$ we obtain, from \rref{140:b},
\begin{equation}
  \label{93}
\dot v_1(t) \leq -\frac{a}{b}\left({k_d}-\frac{ak_ck_\delta}{b}\right)\norm{\vartheta(t)}^2 + k_ck_\delta\frac{{\Delta^*}^2}{b^2}
\end{equation}
for all $t\in [\tau,\tmax)$ and $a\geq a^*$. In view of \rref{11} the factor of $\norm{\vartheta(t)}^2$ is negative.

Next, we proceed by {\em reductio ad absurdum}. Assume that $\norm{x(t)}\to \infty$ as $t\to\infty$ then, either $\norm{\vartheta(t)}$ grows unboundedly as $t\to\infty$ or it is bounded for all $t$. In the first case, since \rref{93} holds for any $\tmax$ by continuity of solutions and since $\Delta^*$ is independent of $\tmax$, we can (if necessary) extend the interval so that for sufficiently large $t\in [\tau,\tmax)$ we have $\norm{\vartheta(t)}^2\geq (\Delta^*/a^*)^2\geq 1$ consequently, from \rref{93},
\begin{equation}
  \label{93.2}
\dot v_1(t) \leq -\left[ \frac{a}{b}\left({k_d}-\frac{ak_ck_\delta}{b}\right)\sfrac{1}{{a^*}^2} - \frac{k_ck_\delta}{b^2}\right]{\Delta^*}^2
\end{equation}
and the factor of ${\Delta^*}^2$ is positive if \rref{11} holds and
\[
\frac{ak_d}{2ba^{*2}} > \frac{k_ck_\delta}{b^2}
\]
which, since $a\geq a^*$,  also holds due to \rref{11}. Therefore, $\dot v_1(t)\leq 0$ which implies that $v_1(t)$ is bounded. Since $V_1$ is radially unbounded we also obtain that $\norm{x(t)}$ is uniformly bounded. Next, assume that $\norm{\vartheta(t)}$ is uniformly bounded for any $t\geq t_\circ\geq 0$ then,  and either $\norm{\dtq(t)}$ or $\norm{\tq(t)}$ (or both) grow unboundedly. If $\norm{\dtq(t)}$ grows unboundedly it follows, in view of \rref{92.5}, that $\norm{\vartheta(t)}\to \infty$ and  the previous reasoning applies again. Finally, consider the case that  $\norm{x(t)}\to \infty$  due to the unbounded growth of $\norm{\tq(t)}$ and consider the function $V_2:\mRp\times\mR^{3n}\to\mRp$, 
    \begin{equation}
      \label{94}
V_2(t,\tq,\dtq,\vartheta) = (\ep_1 \tq - \ep_2\vartheta )^\top D(\tq+\qd(t))\dtq,\quad  \ep_1,\ \ep_2 < 1 
    \end{equation}
which in view of  \rref{100:b} and \rref{100:a}, satisfies
 \begin{align}
\nonumber \dot V_2 = &(\ep_1\tq - \ep_2\vartheta)^\top \left(\,-k_d\vartheta - k_p \tq -[C(q,\dot q) + C(q,\dqd)]\dtq\,\right)
 \\& + \ep_1 \dtq^\top D(q) \dtq - \ep_2( -a\vartheta + b\dtq)^\top D(q) \dtq 
\nonumber \\& \ 
+ (\ep_1\tq-\ep_2\vartheta)^\top \dot {\overparen{D(q)}} \dtq.\label{94.5}
 \end{align}
Let $R$ be an arbitrary positive number and define 
$$\Omega:=\Big\{x\in\mR^{3n}:\tq\in\mR^n,\, \max\big\{\norm{\dtq},\,\norm{\vartheta}\big\}\leq R\Big\}.$$ Then, 
\begin{subequations}\label{95}
  \begin{eqnarray}
\norm{\ep_1\tq^\top C(q,\dot q)^\top \dtq } & \leq & \ep_1 k_c  \norm{\tq}\norm{\dtq} \big[R+k_\delta\big] 
\\
\norm{\ep_2\vartheta^\top C(q,\dot q)^\top \dtq } & \leq & \ep_2 k_c \norm{\vartheta}\norm{\dtq} \big[R+k_\delta\big] 
\end{eqnarray}
\end{subequations}
--see \rref{bndonqd}, which implies that for all $x\in\Omega$, all the terms of undefined sign on the right-hand side of \rref{94.5} may be upper bounded by a first-order polynomial of $\norm{\tq}$. Therefore, using Assumption \ref{ass1} and \rref{95}, we see that there exist positive numbers $c_1$, $c_2$ such that, defining $v_2(t):= V_2(t,\tq(t),\dtq(t),\vartheta(t))$,
 \begin{equation}
   \label{95.5}
\dot v_2(t) \leq -\ep_1k_p\norm{\tq(t)}^2 + c_1\norm{\tq(t)} + c_2
 \end{equation}
for all $t\geq t_\circ$ and $x(t)\in \Omega$ that is,  $\dot v_2(t)$ becomes negative as $\norm{\tq(t)}\to \infty$. 

Next, define $V:\mRp\times\mR^{3n}\to\mR$, 
\begin{equation}
  \label{110}
V(t,x) := V_1(t, \tilde q, \dot {\tilde q},\vartheta) + V_2(t, \tilde q, \dot {\tilde q},\vartheta)
\end{equation}
which  is positive definite for sufficiently large control gains, independently of the initial conditions. To see this, note that defining,
\begin{eqnarray*}
  M_1 := \begin{bmatrix}
  k_p I & 2\ep_1 D\\
2\ep_1 D^\top & D 
\end{bmatrix}\\
M_2 :=\begin{bmatrix}
  \sfrac{k_d}{b} I & -2\ep_2 D\\
-2\ep_2 D^\top & D 
\end{bmatrix},
\end{eqnarray*}
we have 
\begin{eqnarray*}
&& V(t,x) \geq 
\frac{1}{4}
\begin{bmatrix}
\tq \\ \dtq  
\end{bmatrix}^\top\!\!
M_1
\begin{bmatrix}
\tq \\ \dtq  
\end{bmatrix}
+ 
\frac{1}{4}
\begin{bmatrix}
\vartheta \\ \dtq  
\end{bmatrix}^\top\!\!
M_2
\begin{bmatrix}
\vartheta \\ \dtq  
\end{bmatrix}\,
\end{eqnarray*}
and both $M_1$ and $M_2$ are positive semidefinite respectively if
\begin{eqnarray}\label{115}
  \sfrac{k_p}{4d_M}\geq \ep_1^2,\qquad
  \sfrac{k_d}{4bd_M}\geq \ep_2^2\,.
\end{eqnarray}
For any given positive gains $k_p$, $k_d$ and $b$ and the constant $d_M$ one can always find $\ep_1$, $\ep_2>0$ such that the inequalities in \rref{115} hold. From Assumption \ref{ass1}.1 it is also clear that $V$ is proper since $D$ is bounded and $V$ is decrescent, since
\begin{eqnarray}\label{117}
&& V(t,x) \leq 
\begin{bmatrix}
\tq \\ \dtq  
\end{bmatrix}^\top\!\!
M_1
\begin{bmatrix}
\tq \\ \dtq  
\end{bmatrix}
+ 
\begin{bmatrix}
\vartheta \\ \dtq  
\end{bmatrix}^\top\!\!
M_2
\begin{bmatrix}
\vartheta \\ \dtq  
\end{bmatrix}\,
\end{eqnarray}
and the induced norms of $M_1$ and $M_2$ are uniformly bounded from above.

Using \rref{93} and \rref{95.5} we see that $v(t):=V(t,x(t))$ satisfies 
\[
\dot v(t) \leq -\ep_1k_p\norm{\tq(t)}^2 + c_1'\norm{\tq(t)} + c_2' + c_3'\norm{\vartheta(t)}^2
\]
for all $x(t)\in\Omega$, $t\in [\tau,\tmax)$ and appropriate (innocuous) values of $c_1'$, $c_2'$ and $c_3'$. We conclude that if either $\norm{\tq(t)}$,  $\norm{\dtq(t)}$ or  $\norm{\vartheta(t)}$  grows unboundedly, there exists $t\in [\tau,\tmax)$ (if necessary, replace $\tmax$ with $\tmax_{\mbox{\scriptsize new}} > \tmax$) such that $ \dot v(t) \leq 0$. By continuity of the solutions and uniform forward completeness we may extend $[\tau,\tmax)$ to $[\tau,\infty)$ and conclude that $v(t)$ is uniformly bounded on $[\tau,\infty)$. Since $V$ is proper $\norm{x(t)}$ is also uniformly bounded on $[\tau,\infty)$. Using again uniform forward completeness --see Lemma \ref{prop:zero}, we obtain uniform global boundedness on $[t_\circ,\infty)$.
\end{proof}

\begin{lemma}\label{prop:two}
The origin of the system  \rref{10} under the conditions of Theorem \ref{thm:main1} is uniformly (asymptotically) stable. Furthermore, assume that for each $r>0$ there exists $R(r)$ such that if $x(t_\circ)\in B_r$ then $x(t)\in B_R$ for all $t\geq t_\circ$. Under these conditions, the origin is uniformly globally attractive.
\end{lemma}
\begin{proof} 
Let the control gains be fixed according to \rref{11} and for any $r>0$ let the property of uniform global boundedness generate $R=\gamma(r)+c$\,  ---see Definition \ref{def:ugb}, such that $x(t)\in B_R$ for all $t\geq t_\circ$. Consider a function $V:\mRp\times B_R\to\mR$ defined as in \rref{110}. Under Assumption \ref{ass1} its total time-derivative along the trajectories of \rref{100:b}, \rref{100:a} satisfies,  for all $(t,x)\in\mRp\times B_R$,
  \begin{eqnarray}
\nonumber 
\dot V &\!\!\! \leq &\!\!\! -\frac{\ep_1k_p\norm{\tq}^2}{2}-\frac{\ep_2b d_m\norm{\dtq}^2}{2} + \ep_1d_M\norm{\dtq}^2 
\nonumber \\[1.3mm] &   &
- \frac{1}{2}\vecttwo{\tq}{\dtq}^\top 
\mattwo{\ep_1 k_p / \mbox{\small $2$} }{-\ep_1k_c\left(R+k_\delta\right)}{ \ep_2bd_m / \mbox{\small $2$} }
\vecttwo{\tq}{\dtq}
\nonumber \\[1.3mm] &   &
- \frac{1}{2}\vecttwo{\tq}{\vartheta}^\top 
\mattwo{\ep_1 k_p / \mbox{\small $2$} }{ -(\ep_1k_d+\ep_2k_p) }{k_da/ \mbox{\small $2b$} }
\vecttwo{\tq}{\vartheta}
\nonumber \\[1.3mm] &   &
- \frac{1}{2}\shrinkthis{0.97}{\vecttwo{\dtq}{\vartheta}^{\!\!\top} \hspace{-1.8mm}
\mattwo{\sfrac{\ep_2bd_m}{2} }{-\ep_2 \left(k_c (R+k_ck_\delta)+ad_M\right)}{k_da/ \mbox{\small $2b$} }}\hspace{-1.8mm}
\vecttwo{\dtq}{\vartheta}
\hspace{-3mm}\ \nonumber \\  & & \ 
-\left(\dty\frac{k_da}{2b} - \ep_2k_d\right)\norm{\vartheta}^2 + k_ck_\delta\norm{\dtq}^2
\label{179}
\end{eqnarray}
where ``*'' stands for the opposite element in the matrix with respect to the main diagonal. The second and third terms on the right-hand side of the previous inequality may be grouped together then, note that the factor $\big[(\ep_2b d_m/2)-\ep_1 d_M\big]$ of $\dtq $ is positive for sufficiently small values of $\ep_1/\ep_2$. Also, the first matrix above is positive definite if
\begin{eqnarray*}
  \sfrac{\ep_2}{4\ep_1}bd_m \geq \sfrac{ k_c^2\left(R+k_\delta\right)^2}{k_p} + d_M
\end{eqnarray*}
which holds for control gains independent of the initial conditions and of $R$, if 
\begin{equation}\label{180}
  \sfrac{\ep_2}{\ep_1}=\mathcal O\left(R^2\right)\,
\end{equation}
which also imposes $\ep_1/\ep_2$ to be ``small''. 
The second matrix is positive if
\begin{eqnarray*}
   \sfrac{\ep_1k_pk_da}{4b} \geq  (\ep_1k_d + \ep_2 k_p)^2
\end{eqnarray*}
which holds for sufficiently small values of $\ep_1$ and $\ep_2$. Finally, the third matrix is positive definite if
\begin{eqnarray*}
  \sfrac{k_dad_m}{4} \geq  \ep_2 \Big[(R+k_ck_\delta) k_c+k_\delta k_c + ad_M\Big]^2&
\end{eqnarray*}
which is satisfied for sufficiently small values of 
\begin{equation}\label{200}
  \ep_2 = \mathcal O\left(\sfrac{1}{R^2}\right)
\end{equation}
which in turn, in view of \rref{180}, imposes that 
\begin{equation}\label{220}
    \ep_1 = \mathcal O\left(\sfrac{1}{R^4}\right)\,.
\end{equation}
Furthermore, the factor of $\vartheta^2$ is negative if $\ep_2 < {a/2b}$. Note that none of these definitions violate \rref{115} nor they restrict the gains relatively to the value of $R$. Thus, for all  $(t,x)\in \mRp\times B_R$ and for any $R\geq 0$ there exist $\lambda\in(0,1)$ and $c>0$ such that 
\begin{equation}
  \label{190}
\dot V (t,x) \leq -c\norm{x}^2 -\sfrac{\lambda k_d a}{2b}\norm{\vartheta}^2 + k_ck_\delta\norm{\dtq}^2 \quad \forall\, (t,x)\in \mRp\times B_R
\end{equation}
with control gains independent of $R$. At this point it is clear that uniform (asymptotic) stability of the origin follows from fixing $R$ to be ``small''.

To continue further with the proof of uniform global attractivity, we invoke Lemma \ref{lem:kalsi} and uniform global boundedness to see that \rref{92.5} holds for all $t\geq t_\circ$ therefore, from \rref{190},  we have 
\begin{equation}\label{195}
  \dot v(t) \leq  -c\norm{x(t)}^2  -\Big[\sfrac{\lambda k_da}{2b}-\sfrac{k_ck_\delta a^2}{b^2}\Big]\norm{\vartheta(t)}^2 + \sfrac{k_ck_\delta\Delta^{*2}}{b^2} \quad\forall\, (t,x_\circ)\in \mRp\times B_r.
\end{equation}
where $v(t)= V(t,x(t))$. The factor of $\norm{\vartheta(t)}^2$ is non-positive due to \rref{11}, therefore
\begin{equation}\label{eq:33}
  \dot v(t) \leq  -c\norm{x(t)}^2  + \tilde\upsilon \quad\forall\, (t,x_\circ)\in \mRp\times B_r
\end{equation}
where $\tilde\upsilon := k_ck_\delta\Delta^{*2}/b^2$. 

Next, reconsider \rref{117} and let $m_M$ be an upper-bound on the induced norms of $M_1$ and $M_2$ then, $v(t_\circ)\leq 2m_M\norm{x_\circ}^2$. Integrating on both sides of \rref{eq:33} from $t_\circ$ to $t$ we see that \rref{eq:int-lem} holds for any $t\geq t_\circ$ with $\gamma(s)= c|s|^2$ and $\beta_{{r\tilde\upsilon}} = 2m_M r^2$. The proof is completed by observing that the previous computations hold for arbitrary $r$ and $\Delta^*$, hence any $\tilde\upsilon(\Delta^*) >0$, and invoking Theorem \ref{thm:uga}. 
\end{proof}

It is worth to underline the following statement which follows straightforward from the previous proof. It improves the main result in \cite{al:LORORT} in the sense that the conditions on the control gains are significantly relaxed but it also covers all other semiglobal results.
\begin{corollary}
The origin of the closed-loop  system \rref{10} with \rref{90} under condition \rref{11} is semiglobally exponentially stable. 
\end{corollary}
\begin{proof}
  Referring to \rref{190} we see that for any $R$ there exists $c' = c-R^2 >0$, which depends on the control gains, such that $\dot V\leq -c'\norm{x}^2$.  
\end{proof}

From all the above we also conclude the following. 
\begin{corollary}
The origin of the closed-loop  system \rref{10} with \rref{90} under condition \rref{11} is uniformly globally asymptotically stable and exponentially stable on any compact. 
\end{corollary}
Therefore, the origin is uniformly locally exponentially stable, which in turn implies that the state trajectories are uniformly square-integrable. 

\section{Systems with relative degree $m+2$}
\label{sec:extension}
In this section, we consider the output-feedback problem stated in Definition \ref{def:pblm} for systems such that  the control input enters through a chain of $m$ integrators that is,
\begin{subequations}\label{600}
  \begin{eqnarray}\label{600a}
D(q) \ddot q + C(q,\dot q)\dot q + g(q) & = & \opxi_1\\
\nonumber \dot \opxi_1 & = & \opxi_2 \\
\nonumber & \vdots & \\
\dot \opxi_m & = & u \label{600b}
\end{eqnarray}
\end{subequations}
where $u$, $\xi_i\in\mR^n$ for all $i\leq m$. 

The model \rref{600} covers several interesting cases which may be related to other challenging open problems of nonlinear control such as the control of Lagrangian (mechanical) systems, taking into account the {\em actuator} dynamics that is,
\begin{subequations}    \label{603}
  \begin{eqnarray}
    \label{603a}  D(q) \ddot q + C(q,\dot q)\dot q + g(q)   \ =\ \opxi_1\\
      \label{603b}  \dot x \ =\  f(t,x) + \tau, \quad \opxi_1 = h(t,x)
  \end{eqnarray}
\end{subequations}
where $x$ here denotes the actuator's state and $\tau$ the control input. Provided that the actuator dynamics \rref{603b} is input-output (globally) feedback-linearizable with respect to the output $\opxi_1$, the model \rref{603} may be transformed into \rref{600}. Although this task is very difficult in general, there exists a considerable bulk of literature on the subject, particularly for electrical machines --see \cite{MARTOM}. 
 A ``simple'' example concerns  the {\em flexible-joint} robot manipulator model, simultaneously and independently introduced in \cite{BURZAR,MCEJR}, 
\begin{subequations}    \label{595}
  \begin{eqnarray}
    \label{595a}  D(q_1) \ddot q_1 + C(q_1,\dot q_1)\dot q_1 + \g(q_1) & = & K(q_2-q_1)\\
    \label{595b}  J\ddot q_2 + K(q_2-q_1) & = & \tau
  \end{eqnarray}
\end{subequations}
where $q_1$ and $q_2$ denote, respectively, the link and actuator generalized coordinates, $\g$ represents potential forces,  $K$ is the joint-stiffness matrix (positive diagonal), $J$  is the rotor inertia matrix and $\tau$ is the (physical) control input. It is easy to see that this system may be `transformed' in a system of the  form \rref{600} with $m=2$. Let $\opxi_1 = Kq_2$ and $\opxi_2 = K\dot q_2$ then $\dopxi1 = \opxi_2$ and 
\[
\dot \opxi_2 = KJ^{-1} \big[\tau -\opxi_1 + Kq_1\big]
\]
so it suffices to define $\tau = \opxi_1 - Kq_1 + JK^{-1}u$ and $g(q_1)= \mbox{g}(q_1)+Kq_1$. Furthermore, it is clear that the same computation goes through if \rref{595b} contains nonlinear terms, provided that they may be canceled via output feedback.   

Back to \rref{600},  a natural candidate to solve the output-feedback control problem is backstepping control. Following such rationale, let us assume that $\opxi_1$ is a virtual control input to \rref{600a} and suppose that a given output-feedback control law $\opxi_1^*$ globally asymptotically stabilizes the origin $(\tilde q, \dot{\tilde q})=(0,0)$. After Theorem \ref{thm:main1}, the controller \rref{90} achieves this task however, there are two fundamental technical obstacles to carry on a recursive control design for \rref{600}. Firstly, backstepping control design relies on a control Lyapunov function for \rref{600a} with $\opxi_1=\opxi_1^*$ however, the proof of Theorem \ref{thm:main1} does not provide one and as far as we know this remains an open problem --see Section \ref{sec:disc}. Secondly, the successive derivatives of the virtual control inputs starting with $\dopxi1$, depend on the unmeasurable velocities and successive derivatives of the latter. 

Thus, although one cannot expect to apply classical backstepping control as such we follow a similar reasoning, relying on the use of {\em cascaded approximate differentiators} 
\begin{equation}
\label{602}  \vartheta_0 = \frac{b_0 s}{s+a_0}\tilde q,\
\end{equation}
\begin{equation}\label{604}
  \vartheta_i = \frac{b_i s}{s+a_i}{\opxi_i^*}, \quad i\in [1,m]
\end{equation}
and, in lack of Lyapunov's direct method, we carry out an inductive trajectory-based analysis along the proof-lines of Theorem \ref{thm:main1}. 

To that end, we introduce one more hypothesis for the {\em underactuated} system \rref{600}.
\begin{assumption}\label{ass2}
(1) The reference trajectories $t\mapsto \qd$ are solutions to the set of differential equations 
\begin{equation}
  \label{605}
 D(q_d)\ddqd  + C(q_d,\dqd)\dqd + g(q_d) = 0;
\end{equation}
(2) the Coriolis and centrifugal forces matrix function $C$ is globally Lipschitz and bounded in the first argument hence, there exists $\lambda>0$ such that, for each $y\in\mR^n$ 
\[
\norm{C(w,y)-C(z,y)} \leq \lambda\mbox{\rm sat}(\norm{w-z})\norm{y}
\] 
where {\rm sat} denotes a generic saturation function, for instance, {\rm sat}$(w)= \delta_1\mbox{\rm sgn}(w)\min\{\delta_2,|w|\}$ with $\delta_1$, $\delta_2>0$;\\
\noindent (3) the function representing the potential-energy force satisfies 
\begin{equation}
  \label{607}
\exists\ k_v > 0 \ : \ \quad \norm{\frac{\partial g}{\partial q}} \leq k_v \quad \forall\ q\in\mR^n.
\end{equation}
\end{assumption}

In view of Assumption \ref{ass1}, item (1) of Assumption \ref{ass2} implies that $q_d$ is (at least) twice continuously differentiable and the first two derivatives are bounded, hence they satisfy \rref{bndonqd}. The second item slightly reinforces (only) item 2 of Assumption \ref{ass1}; it holds if $C$ is continuously differentiable with a bounded derivative in $w$. The last two items hold for instance, for robot manipulators with only prismatic or only revolute joints, since $w\mapsto C$ and $q\mapsto g$ are defined via constants and trigonometric or linear-growth functions.
\begin{example}\label{rem:rmk3}
For the particular case of flexible-joint robots, where a preliminary control design stage is needed, Assumption \ref{ass2} holds provided that 
\begin{equation}
  \label{608}
 D(q_{1d})\ddot q_{1d} + C(q_{1d},\dot q_{1d})\dot q_{1d} + \g(q_{1d})+ Kq_{1d} = 0
\end{equation}
hence, defining $g(q_1) = \g(q_{1d})+ Kq_{1d}$, inequality \rref{607} holds if $\norm{\frac{\partial \mbox{\small \rm g}}{\partial q}} \leq k_v$ and $K$ is bounded\footnote{Interestingly, the latter holds for free since otherwise, as $\norm{K}\to\infty$ we recover the model \rref{10} --see \cite{SPOVID}.}.
\end{example}

For clarity of exposition we first deal with the case of one added integrator then, in Section \ref{sec:mg1}, we solve the problem for the general case.  

\subsection{Case of one added integrator}
\label{sec:mequal1}
We have $m=1$ hence the actual control input appears in the first integrator {\em i.e.},
\begin{equation}\label{619}
  \dopxi1 = u.
\end{equation}
Let us introduce the control gains $k_{p_i}$, $k_{d_i}$ as positive reals for all $i\in [0,m]$; let the gains with $i=0$ correspond to those in Section \ref{sec:globtrack}. Let $\opxi_1^*$  be defined as on the right-hand side of \rref{90c} with $k_p= k_{p_0}$, $k_d= k_{d_0}$ and $\vartheta=\vartheta_0$ --\cf \rref{602} that is, 
\begin{equation}
\label{610a}\opxi_1^*  =  -k_{p_0} \tilde q - k_{d_0} \vartheta_0 + D(q)\ddot q_d + C(q,\dot q_d)\dot q_d + g(q).
\end{equation}
It is obvious that the control law
\begin{equation}
  \label{620}
u = - k_{p1}\tilde \opxi + \dot\opxi^*_1
\end{equation}
makes the origin of the closed-loop system, $\dot{\tilde \opxi}_1 = -k_{p1}\topxi1$, globally exponentially stable; on the other hand, since  $\opxi_1 = \topxi1 + \opxi_1^*$ we have from \rref{600a} and \rref{602},
\begin{subequations}\label{625}
  \begin{eqnarray}
\label{625a}  && D(q)\ddot {\tilde q} + [\,C(q,\dot q) + C(q,\dot q_d)]\dot {\tilde q} + k_{p_0}  \tilde q + k_{d_0} \vartheta_0\  =\ \tilde \opxi_1 \\
\label{625b} && \dot \vartheta_0\ =\ -a_0\vartheta_0 + b_{0}\dtq.
  \end{eqnarray}
\end{subequations}

In view of Theorem \ref{thm:main1} the origin of the system \rref{625} with zero input $\topxi1=0$ is uniformly globally asymptotically stable therefore, a simple cascades argument establishes uniform global asymptotic stability for \rref{600} in closed loop with \rref{610a} and \rref{620}. Nevertheless, the implementation of \rref{620} requires $\dopxi1^*$ which depends on the unmeasured velocities; indeed\footnote{To avoid a cumbersome notation we drop the argument $(t)$ of $\qd$ and its derivatives and write $\dtq$ in place of $\dot q - \dqd(t)$.}, 
\begin{eqnarray}
    \nonumber \dot \opxi_1^*(t,q,\dot q,\vartheta_0) &=& \big[C(q,\dot q_d)+C(q,\dot q)+C(q,\dot q)^\top\big]\ddot q_d + D(q)q_d^{(3)}  
\\ &&
\quad+ M\dqd + \sfrac{\partial g}{\partial q}^{\top}\dot q 
 - k_{p_0} \dtq - k_{d_0}[-a_0\vartheta_0 + b_0 \dtq]
\label{623}
\end{eqnarray}
where
\begin{eqnarray*}
  M = 
  \begin{bmatrix}
    \dot q^\top \sfrac{\partial c_{d1}}{\partial q} + \ddqd^{\top}\sfrac{\partial c_{d1}}{\partial \dqd} \\
\vdots\\
\dot q^\top \sfrac{\partial c_{dn}}{\partial q} + \ddqd^{\top}\sfrac{\partial c_{dn}}{\partial \dqd}\\
  \end{bmatrix}
\end{eqnarray*}
and $c_{di}$ denotes the $i$-th row of $C(q,\dqd)$. 

As mentioned earlier, the dependence of $\dot\opxi_1^*$ on unmeasured state variables imposes several modifications to the controller previously introduced as well as an alternative to the cascades-based analysis. Firstly, we use approximate differentiation on $\opxi_1^*$ and we redefine the controller as 
\begin{subequations}\label{624}
  \begin{eqnarray}  
  \label{624a}
u &=& - k_{p1}\tilde \opxi_1 + k_{d1} \vartheta_1\\
  \label{624d} \vartheta_1 & = & q_{c1} + b_1 \opxi_1^* + \zeta_1\\
\label{624b}\dot q_{c1} &=& -a_1(q_{c1} + b_1 \opxi_1^* + \zeta_1) \\
\label{624c} \dot \zeta_1 &=& -k_{d1}\tilde \opxi_1.
\end{eqnarray}
\end{subequations}
Note that Equations \rref{624d}, \rref{624b} with $\zeta_1=0$ correspond to the state-representation of the approximate differentiation filter \rref{604} with $i=1$. The introduction of the ``extra'' integrator $\zeta_1$ is motivated by Lyapunov redesign, as it will become clearer below. Indeed, differentiating $\vartheta_1$ in \rref{624d}, using \rref{624b}, \rref{624c} and replacing \rref{624a} in \rref{619}, we obtain 
\begin{subequations}\label{627}
  \begin{eqnarray}
\label{627a} &&\hspace{-34.7mm} \dtopxi1\ =\ -k_{p1}\topxi1 + k_{d1}\vartheta_1 - \dopxi1^*\\ 
\label{627b} &&\hspace{-35.5mm} \dot \vartheta_1\ =\ -a_1\vartheta_1 -k_{d1}\topxi1 + b_1\dot\opxi_1^*.
  \end{eqnarray}
\end{subequations}

The resulting closed-loop system \rref{625}, \rref{627}  no longer forms a cascaded-interconnected system however, several interesting properties deserve to be underlined. Firstly, it is easy to see that in view of the `matching' terms $-k_{d1}\topxi1$ and $k_{d1}\vartheta_1$ the origin of \rref{627} with zero ``input'', $\dopxi1^*= 0$, is globally exponentially stable for any positive values of $k_{p1}$ and $a_1$. Secondly, in view of Assumption \ref{ass2}, $\opxi_1^*(t,\qd,\dqd,0)=0$ hence $ \dopxi1^*(t,\qd,\dqd,0)= \sfrac{d\xi_1^*}{dt}(t,\qd,\dqd,0) = 0$. In addition, invoking Assumption \ref{ass1} we see that  $\dopxi1$ is globally Lipschitz in all state variables, uniformly in $t$ and is bounded in the first two arguments therefore, there exist $\eta_1$,  $\eta_2$ and $\eta_3$ such that 
\begin{equation}
  \norm{\dot\opxi_1^*(t,q,\dot q,\vartheta_0)} \leq  \eta_1\sat(\norm{\tq}) + \eta_2\norm{\dtq} + \eta_3\norm{\vartheta_0}.
\label{650}
\end{equation}
Furthermore, a direct computation shows that the total derivative of 
\[
W_1 = \frac{1}{2}\Big[ \big|\topxi1\big|^2 + \norm{\vartheta_1}^2 \Big]
\]
along the trajectories of \rref{627} yields 
\begin{equation}
  \label{627.5}
\dot W_1 = -k_{p1} \big|\topxi1\big|^2 -a_{1} \norm{\vartheta_1}^2 + \dot\opxi_1^*(t,q,\dot q,\vartheta_0)^\top[b_1\vartheta_1 - \topxi1]
\end{equation}
so we conclude that the origin of the system \rref{627} is exponentially stable if $(\tilde q,\dtq,\vartheta_0)=(0,0,0)$ and otherwise input-to-state stable (with linear gain) from the input $(\tq,\dtq,\vartheta_0)$. In turn, the other two ($n$-dimensional) differential equations corresponding to the error dynamics \rref{625}, form a locally input-to-state stable system from the input $\topxi1$. Actually, following the proof guidelines of Lemma \ref{prop:two}, it may be showed that input-to-state-stability holds semiglobally. 

Even though the previous (fairly intuitive) reasoning in terms of small-gain arguments cannot be formally extended to the global case without a strict Lyapunov function for \rref{625}, we claim that the origin of \rref{627}, \rref{625} is uniformly {\em globally} asymptotically stable. The proof of the following statement follows the trajectory-based rationale of the proof of Theorem \ref{thm:main1}.

\begin{theorem}\label{thm:main2}
Consider the system \rref{625}, \rref{623}, \rref{627}, which corresponds to the closed-loop of \rref{600}, \rref{610a} and \rref{624}. Let Assumptions \ref{ass1}, \ref{ass2} hold and let the control gains satisfy
\begin{subequations}
  \begin{eqnarray}\label{632}
\label{632a}
&\displaystyle  
k_{d_0}\left[\sfrac{a_0}{4b_0} - 1 \right]  > \sfrac{3}{2} + \frac{(k_ck_\delta + 3/2)a_0^2}{b_0^2} & \\
\label{632b}
&\displaystyle k_{p1} > \sfrac{1}{2}\Big(\eta_2^2+\eta_3^2\Big) + 2 &\\
\label{632c}
&\displaystyle   a_{1} > \sfrac{1}{2}\Big(\eta_2^2+ \eta_3^2b_1^2 + 1\Big).&
\end{eqnarray}
\end{subequations}
 Then, the origin is uniformly globally asymptotically stable.
\end{theorem}

\begin{proof}
The state of the closed-loop system \rref{625}, \rref{627} is
\begin{equation}\label{629}
  x = [\tq^\top\quad \dtq^\top\quad\vartheta_0^\top\quad\topxi1^\top\quad\vartheta_1^\top ]^ \top 
\end{equation}
and in view of Assumption \ref{ass2}.1, the origin is an equilibrium.

\noindent {\bf Uniform forward completeness.} Consider the function $V_1$ as defined in \rref{120}, for the system \rref{625} that is, 
\begin{equation}
  \label{640}
  V_1(t,\tq,\dtq,\vartheta_0) = \frac{1}{2}\left(
\dtq^\top D(\tq+\qd(t)) \dtq + k_{p_0} \norm{\tq}^2 + \frac{k_{d_0}}{b_0} \norm{\vartheta_0}^2 \right). 
\end{equation}
The total derivative of ${\mathcal V}_1:= V_1 + W_1$ along the trajectories of \rref{627}, \rref{625} yields
\begin{align}\nonumber
  \dot {\mathcal V}_1 = & -\sfrac{a_0 k_{d_0} }{b_0} \norm{\vartheta_0}^2 + k_ck_\delta\norm{\dtq}^2 +\dtq^\top\topxi1\\
&-k_{p1} \big|\topxi1\big|^2 -a_{1} \norm{\vartheta_1}^2 + \dot\opxi_1^{*\top}\big[b_1\vartheta_1 - \tilde \opxi_1\big]\label{645}
\end{align}
where $\dot \opxi_1^*$ is given in \rref{623}. Using \rref{650} and the triangle inequality, we obtain
\begin{eqnarray}\nonumber
  \dot{\mathcal V}_1 &\leq& -\Big[\sfrac{a_0k_{d_0}}{b_0}-1\Big]\norm{\vartheta_0}^2 + \big[k_ck_\delta + \sfrac{3}{2}\big] \norm{\dtq}^2 \\ 
&& - \left[a_{1}-\sfrac{b_1^2}{2}\big(\eta_2^2+ \eta_3^2\big)\right]\norm{\vartheta_1}^2
-\left[k_{p1}-\sfrac{1}{2}\big(1+\eta_2^2+\eta_3^2\big)\right]\big|\tilde\opxi_{1}\big|^2 + \eta_1\Big[b_1\norm{\vartheta_1} + \big|\tilde\opxi_{1}\big| \Big].
\label{700}
\end{eqnarray}
Proceeding as in the proof of Lemma \ref{prop:zero} --\cf \rref{140:b}, we obtain $\dot {\mathcal V}_1 \leq c\mathcal V_1$ for an appropriate choice of $c>0$ therefore, the closed-loop system is uniformly forward complete that is, the closed-loop trajectories satisfy \rref{eq:110} with an appropriate redefinition of $c_1$ and $c_2$.  

\noindent {\bf Uniform global boundedness.} As in the proof of Lemma \ref{prop:one}, we proceed by contradiction; let  $\norm{x(t)}\to\infty$ as $t\to \infty$. Firstly, if $\topxi1$ and $\vartheta_1$ are uniformly bounded then, in view of \rref{700}, there exists $c>0$ such that 
\begin{eqnarray}
  \dot{\mathcal V}_1 &\leq& -\Big[\sfrac{a_0k_{d_0}}{b_0}-1\Big]\norm{\vartheta_0}^2 + \big[k_ck_\delta + \sfrac{3}{2}\big] \norm{\dtq}^2 + c
\label{710}
\end{eqnarray}
--\cf \rref{140:b}. If on the contrary, $\big|\topxi1(t)\ \vartheta_1(t)\big|\to \infty$ as $t\to\infty$ then\footnote{In view of uniform forward completeness, the solutions may either be $uniformly$ bounded or grow unboundedly, also uniformly.}, for sufficiently large $t$ the last three terms on the right-hand side of \rref{700} become non-positive and \rref{710} holds with $c=0$. Arguing as in the first part of the proof of Lemma \ref{prop:one} we conclude that the same arguments as for $v_1(t)$ in \rref{93}, hold for $\mathcal V_1$ along closed-loop trajectories. That is, let Lemma \ref{lem:kalsi} generate, for any $\Delta^*$, a number $a^*(\Delta^*)$ such that \rref{92.5} holds with $b=b_0$ and $a=a_0\geq a_0^*$. Then, similarly to \rref{93} we have
\begin{eqnarray}
  \dot{\mathcal V}_1(t,x(t)) &\leq& -\sfrac{a_0}{b_0}\Big[k_{d_0} - \sfrac{b_0}{a_0}- \sfrac{a_0}{b_0}\big(k_ck_\delta+\sfrac{3}2\big) \Big]\norm{\vartheta_0(t)}^2 + \big[k_ck_\delta + \sfrac{3}{2}\big] \frac{\Delta^{*2}}{b_0^2} + c
\label{712}
\end{eqnarray}
for all $t\in [\tau,\tmax)$ and $a_0\geq a_0^*$ with $\tmax$ defined as in the proof of Lemma \ref{prop:one}. Proceeding further as in the proof of the latter, we argue that if $\norm{\vartheta_0(t)}\to \infty$ then a bound like \rref{93.2}, modulo obvious modifications (in the notation) to the factor of $\Delta^*$, holds for $\mathcal V_1(t,x(t))$; therefore, the solutions are uniformly globally bounded. If, on the contrary,  $\vartheta_0$ is uniformly bounded so is $\dtq$ since the latter is the input to the strictly proper and stable filter \rref{625b} with output $\vartheta_0$ and finite DC gain --see \rref{92.5}. Now assume that $\norm{\tq(t)}\to\infty$ and consider the function $V_2$ defined in \rref{94} by replacing $\vartheta$ with $\vartheta_0$; its time derivative yields $\dot V_2 = Y_2 + \topxi1^\top(\ep_1\tq-\ep_2\vartheta_0)$ where $Y_2$ corresponds to the right-hand-side of \rref{94.5}, modulo a few obvious modifications in the notation. Since all signals are uniformly globally bounded, except eventually $t\mapsto\tq$, it follows that \rref{95.5} holds for the trajectories generated by \rref{625}, \rref{627} with an appropriate (innocuous) redefinition of $c_1$ and $c_2$ and with $k_p = k_{p_0}$. The rest of the proof of Lemma \ref{prop:one} that is, the arguments below \rref{95.5}, continue to hold  {\em mutatis mutandis} and we conclude that the solutions of \rref{627}, \rref{625} are uniformly globally bounded.

\noindent  {\bf Uniform global attractivity}. We proceed as in the proof of Lemma \ref{prop:two}. In view of uniform global boundedness for each $r>0$ there exists $R(r)$ such that if $x(t_\circ)\in B_r$ then $x(t)\in B_R$ for all $t\geq t_\circ$ where $x$ is defined in \rref{629}.

Consider the function $\mathcal V:\mRp\times B_R\to\mR$ defined by 
\[
\mathcal V(t,x) = V_1(t,\tq,\dtq,\vartheta_0) + V_2(t,\tq,\dtq,\vartheta_0)+ W_1(\tilde \opxi_1,\vartheta_1) 
\]
where $V_1$ and $V_2$ are defined as in \rref{640} and \rref{94} respectively. With some obvious modifications, $\dot V_1 + \dot V_2 -\topxi1^\top \big[\dtq +\ep_1\tq-\ep_2\vartheta_0\big]$ satisfies \rref{179} hence, in view of Assumptions \ref{ass1} and \ref{ass2}, the condition\footnote{Note that \rref{632a} implies \rref{11} hence we may proceed as in the proof of Lemma \ref{prop:two}.} \rref{11}, and after the proof of Lemma  \ref{prop:two}, we obtain 
\begin{equation}
  \label{715}
\dot V_1 + \dot V_2 \leq -\sfrac{1}{2}\left[\ep_1k_{p_0}\norm{\tq}^2 + \ep_2b_{0}d_m\norm{\dtq}^2\right] -  \left[\sfrac{k_{d_0}a_0}{4b_0} - \ep_2k_{d_0}\right]\norm{\vartheta_0}^2 +k_ck_\delta\norm{\dtq}^2 + \topxi1^\top \big[\dtq +\ep_1\tq-\ep_2\vartheta_0\big]
\end{equation}
--\cf \rref{179}. Applying the triangle inequality to \rref{715} we obtain
\begin{equation}
  \label{716}
  \begin{split}
  \dot V_1 + \dot V_2 \leq -\sfrac{1}{2}&\left[\ep_1\big(k_{p_0}-1\big)\norm{\tq}^2 + \ep_2b_{0}d_m\norm{\dtq}^2\right]  -  \left[\sfrac{k_{d_0}a_0}{4b_0} - \ep_2\Big(k_{d_0}+\sfrac{1}{2}\Big)\right]\norm{\vartheta_0}^2  \\ & +\big(k_ck_\delta+\sfrac{1}{2}\big)\norm{\dtq}^2 + \sfrac{1}{2}\big(1+\ep_1+\ep_2\big)\big|\topxi1\big|^2.
  \end{split}
\end{equation}
On the other hand, $\dot W_1$ satisfies \rref{627.5} and in view of \rref{650} the last term on the right-hand side of \rref{627.5} is bounded from above by \tonio{we need to define these numbers $\eta$ ... }
\begin{align*}
  \Big[ \eta_1\sat(\norm{\tq}) + \eta_2\norm{\dtq} + \eta_3\norm{\vartheta_0}\Big]&\Big[b_1\norm{\vartheta_1} + \big|\topxi1\big|\Big]\\ 
&\hspace{-1in} \leq 
\sfrac{1}{2}\left[\eta_1^2(b_1^2 + 1)\sat(\norm{\tq})^2 + \left[1+\eta_2^2b_1^2+\eta_3^2b_1^2\right]\norm{\vartheta_1}^2+\left[1+\eta_2^2+\eta_3^2\right]\big|\topxi1\big|^2 + 2\norm{\dtq}^2 +  2\norm{\vartheta_0}^2 \right].
\end{align*}
Next, define 
\begin{eqnarray*}
  \alpha_1& =& \sfrac{k_{d_0}a_0}{4b_0} - \ep_2\Big(k_d +\sfrac{1}{2}\Big) - 1 
\\  \alpha_2 & =& k_{p1} - \sfrac{1}{2}\Big(\eta_2^2+\eta_3^2+2+\ep_1+\ep_2\Big)
\\  \alpha_3 & =& a_{1} - \sfrac{1}{2}\Big(\eta_2^2+ \eta_3^2b_1^2 + 1\Big)
\\  \alpha_4 & =& k_ck_\delta+\sfrac{3}{2}
\end{eqnarray*}
and, for each $\mu > 0$, $\alpha_5(\mu)  = \eta_1^2(b_1^2+1)\sat(\mu)^2 /2$ then
\[
\mathcal V(t,x) \leq -\sfrac{1}{2}\left[\ep_1k_{p_0}\norm{\tq}^2 + \ep_2b_{0}d_m\norm{\dtq}^2\right]-\alpha_1 \norm{\vartheta_0}^2 -\alpha_2\big|\topxi1\big|^2 -\alpha_3 \norm{\vartheta_1}^2 + \alpha_4\norm{\dtq}^2 + \alpha_5(\norm{\tq}).
\]
Next, let Lemma \ref{lem:kalsi} generate $\Delta^*$, $a^*$ as in the proof of Lemma \ref{prop:one} such that for all $a_0\geq a^*$ and all $t\geq t_\circ$, \rref{92.5} holds with $b=b_0$ and $a=a_0$. Then, let uniform global boundedness generate the existence of $\mu$ such that $\norm{\tilde q}\leq \mu$. For all $t\geq t_\circ\geq 0$, we define
\[
c = \min\big\{\ep_1k_{p_0},\ \ep_2b_{0}d_m,\ \alpha_2,\ \alpha_3\big\},\qquad z^\top = \big[\tq^\top\ \dtq^\top\ \tilde \opxi_1^\top \ \vartheta_1^\top\big].
\]
We have, for all   $t\geq t_\circ\geq 0$, all $r>0$ and all $x_\circ\in B_r$,
\[
\mathcal V(t,x(t)) \leq -c\big|z(t)\big|^2 - \Big[\alpha_1 - \sfrac{\alpha_4a_0^2}{b_0^2}\Big]\norm{\vartheta_0(t)}^2 + \Big[\alpha_5(\mu) + \sfrac{\alpha_4\Delta^{*2}}{b_0^2}\Big].
\]
Hence, in view of \rref{632}, the trajectories satisfy a bound like \rref{eq:33} that is,
\begin{equation}\label{745}
  \dot{ \mathcal V}(t,x(t)) \leq  -c'\norm{x(t)}^2  + \tilde\upsilon \quad\forall\, (t,x_\circ)\in \mRp\times B_r
\end{equation}
where $c'>0$ and $\tilde\upsilon(\mu,\Delta^*) = \alpha_5(\mu) + \sfrac{\alpha_4\Delta^{*2}}{b_0^2}$ is uniformly bounded in $\mu$. As in the proof of Lemma \ref{prop:two} we see that from uniform global boundedness of solutions there exists $\beta_{r\tilde\upsilon}>0$ such that 
\[
\int_{t_\circ}^t \big[c'\norm{x(s)}^2 - \tilde\upsilon \big]ds \leq \beta_{r\tilde\upsilon}  < \infty 
\]
for all $(t,x_\circ)\in \mRp\times B_r$ and for any $\tilde\upsilon(\mu,\Delta^*)$ that is, for any pair $(\mu,\Delta^*)$. The result follows invoking Theorem \ref{thm:uga}.
\end{proof}

\subsection{Case of $m>1$ added integrators}
\label{sec:mg1}
Now that we have established uniform global asymptotic stability of the origin of the closed-loop system \rref{625}, \rref{627} we generalize the controller \rref{624} to the case $m>1$. As previously explained, we follow a backstepping-like design with approximate differentiators and we use inductive analysis. Accordingly, let the reference to the second integrator, $\opxi_2^*$, correspond to the right-hand side of \rref{624a} and let the subsequent references $\opxi_{i>2}^*$ be defined in a similar manner that is, 
\begin{subequations}\label{750}
  \begin{eqnarray}  
   \label{750aa} \opxi^*_{i+1} &=& - k_{p_i}\tilde \opxi_i + k_{d_i} \vartheta_i \hspace{18.5mm} \forall\ i\in [1,m-1]\\
&&\hspace{-16mm}\left.     
\begin{array}{rcl}
       \label{750b} \vartheta_i & = & q_{ci} + b_i \opxi_i^* + \zeta_i\\
  \label{750c}\dot q_{ci} &=& -a_i(q_{ci} + b_i \opxi_i^*+ \zeta_i) \\
  \label{750d} \dot \zeta_i &=& -\big(k_{d_i} -\sigma_i\big)
  \tilde \opxi_i\\     
     \end{array} \right\} \quad\forall\ i\in [1,m]\\
  \label{750a} u &=& - k_{p_m}\tilde \opxi_m + k_{d_m} \vartheta_m
\end{eqnarray}
\end{subequations}
where $\sigma_i$ is a control redesign constant gain defined as
\[
\sigma_1 = 0, \quad \sigma_i = {b_ik_{p_{i-1}}}\ \, \forall\, i\in [2,m].
\]
 Therefore, using $\dot{\opxi}_i= {\topxi{i+1}} + \opxi^*_{i+1} \pm \dot\opxi_i^*$, we obtain
\begin{subequations}\label{760}
\begin{eqnarray}
\label{760a} && \dot{\tilde\opxi}_i\ =\ -k_{p_i}\topxii + k_{d_i}\vartheta_i - \dot\opxi_i^* + \tilde\opxi_{i+1} \qquad\forall\ i\in [1,m-1]\\ 
&& \dot{\tilde\opxi}_m\ =\ -k_{p_m}\tilde\opxi_m + k_{d_m}\vartheta_m - \dot\opxi_m^* \\
\label{760b} && \dot \vartheta_i\ =\ -a_i\vartheta_i -(k_{d_i}-\sigma_i)\topxii + b_i\dot\opxi_i^*\qquad \forall\ i\in [1,m]
  \end{eqnarray}
\end{subequations}

As for the case of one integrator, it is convenient to underline a few relevant features of the error equations \rref{760}. Firstly, and merely for the purpose of discussion, let $\dopxi{i}^*=0$ in \rref{760} for all $i\leq m$. Then, for each $i<m$, the dynamics \rref{760} is driven by the state of the $[i${\small +1}$]$th system in a {\em cascade} configuration. Implicitly labeled $\Sigma_i$, these blocks are illustrated in Figure \ref{fig:block} (see the shadowed block on the left, $\Sigma_m$) which depicts the error dynamical system.
\begin{figure}
  \input{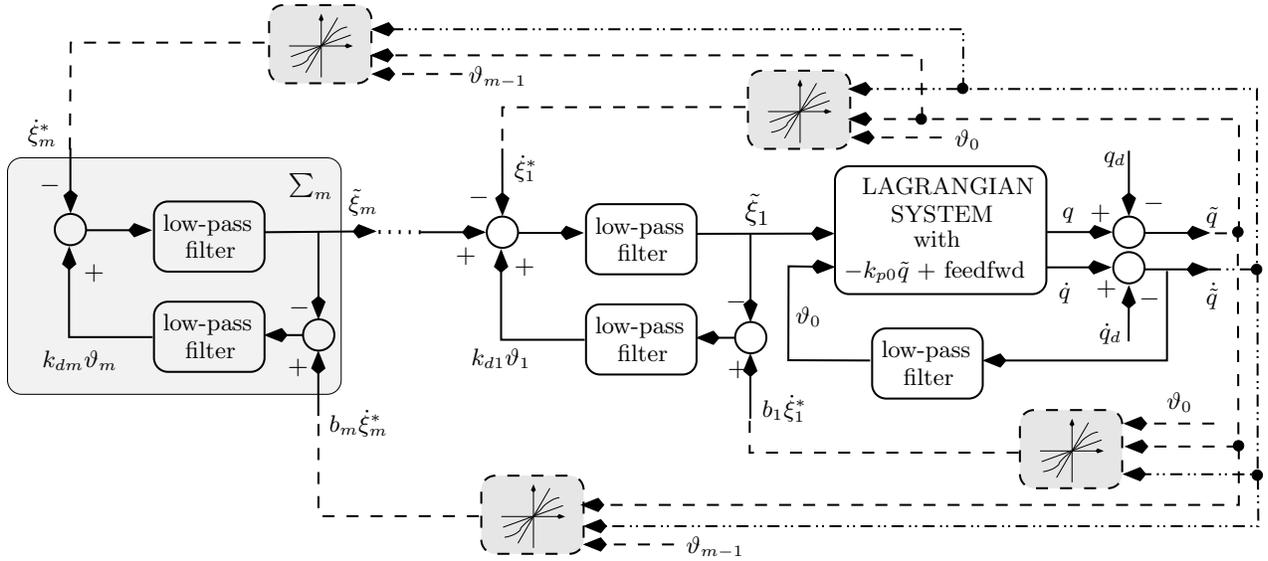}
  \caption{Schematic block-diagram of the error dynamics, {\em not} the control {\em implementation}. The nested interconnections represented by dashed lines are illustrated for analysis purpose. }
\label{fig:block}
\end{figure}
Each of these blocks $\Sigma_i$ is composed by two low-pass filters which are feedback-interconnected; the two filters are finite-gain strictly passive systems therefore, input-output stable and exponentially stable if the ``exogenous'' input $\dopxi{i}^*=0$. The last block (on the right) in the chain corresponds to the system defined by Eqs. \rref{625} that is, the Lagrangian system with the P$^+$ control law $-k_{p_0} \tilde q + D(q)\ddot q_d + C(q,\dot q_d)\dot q_d + g(q)$ --\cf \rref{90c} as input, feedback-interconnected with the approximate differentiator \rref{602}. Notice that this block consists in two passive systems feedback-interconnected and, in view of Theorem \ref{thm:main1}, it is input-to-state stable with respect to the input $\topxi1$. Thus, neglecting  the additional inputs $\dopxi{i}^*$ (that is, making abstraction of the interconnections represented by dashed lines in Fig. \ref{fig:block}) the overall closed-loop system may be considered as the {\em cascade} interconnection of (asymptotically) stable systems. 

Now we take into account $\dopxi{i}^*$ as additional input to each block $\Sigma_i$, implicitly illustrated in Figure \ref{fig:block} and defined by Eqs. \rref{760}. As for the case when $m=1$, $\dopxi{i}^*$ depends on the Lagrangian coordinates, thereby closing a feedback loop for each block $\Sigma_i$ and disrupting the cascaded structure. This is illustrated via dashed lines in Figure \ref{fig:block}; as it may be appreciated in the latter, the overall closed-loop system consists in a series of {\em nested} feedback loops in which the inner-most loop is defined by the Eqs. \rref{625} and the outer-most loop closes via  $\dot\opxi^*_m$. In the case that $m=1$ we recover the system analyzed in Section \ref{sec:mequal1} for which we have established uniform global asymptotic stability. Analogously to this case, if $m>1$ we are guided by a {\em nested} small-gain-based reasoning to establish the same property for systems of higher relative degree. However, since we do not dispose of Lyapunov functions for each block, our formal analysis is trajectory-based.

To that end, we start by collecting the scalar gains $a_i$, $b_i$, $k_{d_i}$ in diagonal matrices $A$, $B$ and $K_d$ respectively, to define 

$$K_p = 
\begin{bmatrix}
  k_{p_1} &\!\! -1 &\! 0 &\!\!\! \cdots &\!\! 0\\
   0     &\! \ddots &\ \ddots &\!\!\!  &\!\! \vdots\\
   \vdots &\!  &\!  &\!  &\! -1\\
   0     &\! \cdots &\! \cdots &\!\!\! 0 &\!\! k_{p_m}\\
\end{bmatrix},\quad 
\tilde \opxi =
\begin{bmatrix}
  \topxi1\\  \vdots\\  \tilde\opxi_m
\end{bmatrix}\quad
\mbox{and }\quad
\vartheta=\begin{bmatrix}
  \vartheta_1\\  \vdots\\  \vartheta_m
\end{bmatrix}.
$$
Then, the overall closed-loop system becomes 
  \begin{eqnarray}
\label{765}  &&
\left\{ \begin{array}{l}
  D(q)\ddot {\tilde q} + [\,C(q,\dot q) + C(q,\dot q_d)]\dot {\tilde q} + k_{p_0}  \tilde q + k_{d_0} \vartheta_0  =  \topxi1 \\
 \dot \vartheta_0\  =\ -a_0\vartheta_0 + b_{0}\dtq                                    
\end{array}\right. \\[2mm]
&&\left\{ \begin{array}{rcl}
\dtopxi{}& =& -K_{p}\tilde\opxi + K_{d}\vartheta - \dot\opxi^*  \\ 
\dot \vartheta & =& -A\vartheta -K_{d}\tilde\opxi + B\dot\opxi^*
\end{array} \right.
\label{770}  \end{eqnarray}
which has the same structure as \rref{625}, \rref{627}. Note that $-K_p$ is Hurwitz for any positive values of $k_{p_i}$; moreover, $K_p+K_p^\top$ is positive definite if  $k_{p_i}k_{p_{i+1}} > 1$ for all $i\leq m-1$. Furthermore, $\dot\opxi^*$ is a function of $(t,q,\dot q,\tilde\opxi,\vartheta)$ and, as we have seen for  $\dot\opxi_1^*$ in the previous section, it may be showed that under Assumptions \ref{ass1} and \ref{ass2} it follows that $\dot\opxi^*(t,q,\dot q,\tilde\opxi,\vartheta)$ is globally Lipschitz in all variables, uniformly in $t$, moreover $\dot\opxi^*(t,q_d,\dqd,0,0)\equiv 0$. In other words,  for the sake of analysis and following the small-gain rationale previously discussed, we may consider $\dot\opxi^*$ as an input to \rref{770} which consists in {\em output injection} terms that depend on $\tilde \opxi$, $\vartheta$ and terms that depend on the ``exogenous'' inputs that depend on $\tq_0$, $\dtq_0$, and $\vartheta$, which are generated by \rref{765}. 

To see this more clearly, we develop $\dot \opxi_{i+1}^*$ as a function of $\topxi{}$, $\vartheta$ and $\dot \opxi_{1}^* $. Note that from \rref{750aa}, \rref{760} we have, for all $i\in[1,m-1]$,
\tonio{
\[
\dot \opxi_{i+1}^* = -k_{p_i} \big[-k_{p_i}\topxii + k_{d_i}\vartheta_i - \dot\opxi_i^* + \tilde\opxi_{i+1} \big] + k_{d_i}\big[ -a_i\vartheta_i -k_{d_i}\topxii + b_i\dot\opxi_i^*\big]. 
\]
}
\[
\dot \opxi_{i+1}^* = \big[k_{p_i}^2 -k_{d_i}^2+k_{d_i}\sigma_i\big] \topxii - k_{d_i}\big[k_{p_i} + a_i\big] \vartheta_i - k_{p_i} \tilde\opxi_{i+1} + \big[k_{p_i} + k_{d_i}b_i\big]\dot\opxi_i^*
\]
hence, in view of the recursive definition of $\dot \opxi_{i}$ and the linearity in $\tilde\opxi_{i-1}$ and $\vartheta_{i-1}$, a direct albeit long computation --see the Appendix, shows that for all $i\in[2,m]$,
\[
 \dot \opxi_{i}^* = -k_{p_{i-1}}\tilde\opxi_{i} + \sum_{k=1}^{i-1} \eta_{i_k}\tilde\opxi_{k} - \mu_{i_k}\vartheta_k + \prod_{j=1}^{i-1} \beta_j \dot \opxi_1^*
\]
where we have defined, for all $i\in[2,m]$ and $k\in [1,i-1]$ 
\begin{eqnarray}
\label{771}&\displaystyle \beta_j =  k_{p_{j}} +b_{j}k_{d_{j}},\quad j\in [1, m-1]& \\
&\displaystyle \eta_{i_k} = \prod_{j=k+1}^{i}\beta_{j}^{\mbox{\scriptsize sgn}(i-j)}\big[k^2_{p_{k}} - k^2_{d_{k}}-k_{p_{k}}k_{p_{k-1}}\mbox{\rm sgn}(k-1)\big],
\label{771.5} & \\ &\displaystyle 
\mu_{i_k} = \prod_{j=k+1}^{i}\beta_{j}^{\mbox{\scriptsize sgn}(i-j)}k_{d_{k}}\big[k_{p_{k}} + a_{k}\big].&
\notag\end{eqnarray}
Therefore, 
\begin{equation}\label{772}
\dot \opxi^* = \Gamma_1 \tilde \opxi + \Gamma_2  \vartheta + \Gamma_3 \dot\opxi_1^*
\end{equation}
with
\[
\Gamma_1 =
\begin{bmatrix}
        0 &     0    &   \cdots & 0 \\
\eta_{\shrinkthis[0.6]{0.7}{2}_1} & -k_{p_1} &          & \vdots \\
\vdots    &          &  \ddots & \\
\eta_{\shrinkthis[0.6]{0.7}{m}_1} & \eta_{\shrinkthis[0.6]{0.7}{m}_2} & \cdots & -k_{p_{m-1}} 
\end{bmatrix}
\quad 
\Gamma_2 =
\begin{bmatrix}
        0 &     0    &   \cdots & 0 \\
-\mu_{\shrinkthis[0.6]{0.7}{2}_1} &  0 &          & \vdots \\
\vdots    &          &  \ddots & \\
-\mu_{\shrinkthis[0.6]{0.7}{m}_1} & -\mu_{\shrinkthis[0.6]{0.7}{m}_2} & \cdots & 0
\end{bmatrix}
\quad 
\Gamma_3 =
\begin{bmatrix}
1\\ \beta_1 \\ \beta_1\beta_2\\\vdots
\end{bmatrix}.
\]
We see that Eqs. \rref{770} take the form 
\begin{subequations}\label{778}
  \begin{eqnarray}
  \dtopxi{}& = & -(K_{p}+\Gamma_1)\tilde\opxi + (K_{d}-\Gamma_2)\vartheta - \Gamma_3\dopxi{1}^*  \\ 
  \dot \vartheta & =& -(A-B\Gamma_2)\vartheta - (K_{d}-B\Gamma_1)\tilde\opxi + B\Gamma_3\dopxi{1}^*.
\end{eqnarray}
\end{subequations}

Provided that one choose $A$, $K_p$, $K_d$ and $B$ to render the origin of \rref{778}, with $\dopxi{1}^*=0$, globally exponentially stable we obtain input-to-state stability of \rref{778} with respect to the input $\dopxi{1}^*$. Moreover, we recall that $\dot \opxi_1^*(t,q,\dot q,\vartheta_0)$ is globally Lipschitz in $q$, $\dot q$ and $\vartheta_0$, uniformly in $t$ and satisfies $\dot\opxi_1^*(t,\qd,\dqd,0)\equiv 0$ therefore, one might also establish input-to-state stability with respect to the input $(t,\tq,\dtq,\vartheta_0)$. On the other hand, the system \rref{765} is (at least locally) input-to-state stable with respect to the input $\topxi{1}$. Thus, it is reasonable to conjecture that a small-gain argument applies to the closed-loop system \rref{765}, \rref{770} or equivalently to the interconnected system \rref{765}, \rref{778}, to conclude that the origin is uniformly globally asymptotically stable. Unfortunately, as for the case of one integrator, the proof of this conjecture relies on ISS Lyapunov functions, which we do not dispose of.  Instead, the proof of the following statement relies on an inductive trajectory-based analysis built upon the proof of Theorem \ref{thm:main1} along the proof-lines of Theorem \ref{thm:main2}.

\begin{theorem}\label{thm:main3}
Consider the system \rref{600} and the output-feedback dynamic controller \rref{610a}, \rref{750} under Assumptions \ref{ass1}, \ref{ass2}. Let 
\begin{eqnarray}
\notag 
\mathcal A  = 
\begin{bmatrix}  \hspace{2mm} 
   - k_{p_1}       &   1     &          0          &   0     &  k_{d_1}   &  0  &   \cdots     &   \cdots  &   0  \\ 
  -\eta_{\shrinkthis[0.6]{0.7}{2}_1}      & \hspace{0mm}\shrinkthis[1]{0.9}{-[k_{p_2}\hspace{-1mm}-k_{p_1}]}  &         &    &      \mu_{\shrinkthis[0.6]{0.7}{2}_1}       &  \ddots    &   &    &     \\ 
   \vdots         &      -\eta_{\shrinkthis[0.6]{0.7}{3}_2}       &       \ddots        &         &         \vdots     &    \mu_{\shrinkthis[0.6]{0.7}{3}_2}     & \quad        &                  \\ 
&&& 1&&\vdots&&\ddots&0 \\
  -\eta_{\shrinkthis[0.6]{0.7}{m}_1}      &    -\eta_{\shrinkthis[0.6]{0.7}{m}_2}       &       \cdots        &    \shrinkthis[1]{0.9}{-[k_{p_m}\hspace{-2mm}-k_{p_{m-1}}]}     &      \mu_{\shrinkthis[0.6]{0.7}{m}_1}       &       \mu_{\shrinkthis[0.6]{0.7}{m}_2}      &      \cdots       &    \shrinkthis[1]{0.9}{\mu_{\shrinkthis[0.6]{0.7}{m}_{m-1}} } & k_{d_m}         \\[2mm]
  -k_{d_1}  &  0   &  \cdots  &  0  &  -a_1  &  0   &  \cdots &  &  0    \\ 
  \shrinkthis[1]{0.9}{b_2\eta_{\shrinkthis[0.6]{0.7}{2}_1}} & -k_{d_2}  &    &   &  \shrinkthis[1]{0.9}{-b_2\mu_{\shrinkthis[0.6]{0.7}{2}_1} }  &  \ddots  &  &  &   \vdots   \\ 
   \vdots  & \shrinkthis[1]{0.9}{b_3\eta_{\shrinkthis[0.6]{0.7}{3}_2}} &  \ddots  &  &  \vdots & \shrinkthis[1]{0.9}{-b_3\mu_{\shrinkthis[0.6]{0.7}{3}_2}}  & &   &                  \\ 
&&& 0&&\vdots&&\ddots & 0 \\
  \shrinkthis[1]{0.9}{b_m\eta_{\shrinkthis[0.6]{0.7}{m}_1} }     &    \shrinkthis[1]{0.9}{b_m\eta_{\shrinkthis[0.6]{0.7}{m}_2} }      &       \cdots        &    -k_{d_m}     &      \shrinkthis[1]{0.9}{-b_m\mu_{\shrinkthis[0.6]{0.7}{m}_1} }      &       \shrinkthis[1]{0.9}{-b_m\mu_{\shrinkthis[0.6]{0.7}{m}_2}}      &      \cdots       &   \shrinkthis[1]{0.9}{-b_m\mu_{\shrinkthis[0.6]{0.7}{m}_{m-1}} } &   -a_m         
 \end{bmatrix} \hspace{0mm}\kronecker I
\end{eqnarray}
\begin{eqnarray}
\notag 
\mathcal B = 
\begin{bmatrix}
  -I\quad
  -\beta_1I\quad
  -\shrinkthis[1]{0.9}{\beta_1\beta_2}I\quad
  \cdots\quad
  -\shrinkthis[1]{0.9}{\displaystyle\prod_{j=1}^{m-1}\beta_j}I\quad
  b_1I\quad
  \shrinkthis[1]{0.9}{b_2\beta_1}I\quad
  \shrinkthis[1]{0.9}{b_3\beta_1\beta_2}I\quad
  \cdots\quad 
  \shrinkthis[1]{0.9}{b_m\displaystyle\prod_{j=1}^{m-1}\beta_j}I\quad
\end{bmatrix}^\top.
\end{eqnarray}
Let the control gains be such that $\mathcal A$ is Hurwitz, 
\begin{equation}
  k_{d_0}\left[\sfrac{a_0}{4b_0} - m \right]  > \sfrac{m+2}{2} + \frac{[k_ck_\delta + (m+2)/2]a_0^2}{b_0^2}\label{lego}
\end{equation}
 and there exist positive definite matrices $Q$ and $P$ such that $Q= -\mathcal A^\top P - P \mathcal A$ and 
\begin{equation}\label{786}
  Q > (\eta_2^2+\eta_3^2)\mbox{\rm diag}\big\{ \norm{ [ P\mathcal B]_i }^2\big\}
\end{equation}
where $[ P\mathcal B]_i$ with $i\in \{1\ldots 2m\}$ denotes the $i$th $n\times n$ block of $P\mathcal B$. Then, the origin of the closed-loop system is uniformly globally asymptotically stable. 
\end{theorem}
\noindent {\bf Sketch of proof.} The proof follows {\em mutatis mutandis} that of Theorem \ref{thm:main2}. Defining $\opchi:= [\topxi{}\ \vartheta]^\top$ and the closed-loop equations \rref{778} become 
\begin{equation}
  \label{787}\dopchi = \mathcal A \opchi + \mathcal B \dot \opxi_1^*.
\end{equation}
Then, the total derivative of 
\begin{equation}
   W(\opchi) = \sfrac{1}{2}\opchi^\top P \opchi 
\label{789}
\end{equation}
along the trajectories generated by \rref{787}  satisfies
\begin{equation}
  \dot W \leq -\sfrac{1}{2}\opchi^\top Q\, \opchi + \opchi^\top P\mathcal B\, \dot \opxi_1^* .
\label{790}
\end{equation}
On the other hand, using \rref{650} we obtain
\begin{eqnarray*}
  \opchi^\top P\mathcal B\, \dot \opxi_1^* & = & \left(\sum_{i=1}^m\topxi{i}^\top[ P\mathcal B]_i + \vartheta_i^\top[ P\mathcal B]_{m+i}\right)\dot\opxi^*_1\\
&& \leq \ \frac{1}{2}\left(\sum_{i=1}^m\norm{\topxii}^2\norm{[ P\mathcal B]_i}^2+ \norm{\vartheta_i}^2\norm{[ P\mathcal B]_{m+i}}^2  \right)(\eta_2^2+\eta_3^2)\\ &&\qquad +\, \eta_1\left(\sum_{i=1}^m\norm{\topxii}\norm{[ P\mathcal B]_i} + \norm{\vartheta_i}\norm{[ P\mathcal B]_{m+i}}  \right) + \left(\norm{\dot {\tilde q}}^2 + \norm{\vartheta_0^2}\right)m.
\end{eqnarray*}
hence the total derivative of 
\[
\mathcal V(t,\opchi) =  \frac{1}{2}\left(
\dtq^\top D(\tq+\qd(t)) \dtq + k_{p_0} \norm{\tq}^2 + \frac{k_{d_0}}{b_0} \norm{\vartheta_0}^2 \right)
+ W(\opchi)
\]
yields
\begin{eqnarray*}
  \dot {\mathcal V} &\leq & -\Big[\sfrac{a_0k_{d_0}}{b_0}-m\Big]\norm{\vartheta_0}^2 + \big[k_ck_\delta + \sfrac{m+2}{2}\big] \norm{\dtq}^2 
- \opchi^\top \left[Q- (\eta_2^2+\eta_3^2)\mbox{diag}\big\{ \norm{ [ P\mathcal B]_i }^2\big\}\right]  \,\opchi 
\\  && \quad
+\, \eta_1\left(\sum_{i=1}^m\norm{\topxii}\norm{[ P\mathcal B]_i} + \norm{\vartheta_i}\norm{[ P\mathcal B]_{m+i}}  \right).
\end{eqnarray*}
By assumption, the quadratic term in $\opchi$ is negative definite. The rest of the proof follows as for Theorem \ref{thm:main2} --\cf \rref{700}.
\qed

Although Theorem \ref{thm:main3} makes a general statement, the obvious drawback is the lack of explicit conditions on the control gains, which ensure uniform global asymptotic stability. The following proposition gives sufficient conditions on the control gains for the conditions of Theorem \ref{thm:main3} to be satisfied. 

\begin{proposition}\label{prop1}
Consider the system \rref{600} and the output-feedback dynamic controller  \rref{602}, \rref{610a} and \rref{750} under Assumptions \ref{ass1} and \ref{ass2}. Let the control gains be such that \rref{632} holds, the matrix $Q\,-$\,\textonehalf\,diag$\{Q\}$, where diag$\{Q\}$ corresponds to the main diagonal of 

\begin{picture}(220,220)
\put(82,112){\psframe*[linecolor=mygray](-1,-.2)(5,3.3)}
\put(233,28){\psframe*[linecolor=mygray](.1,-.7)(7.45,2.8)}
\put(340,170){{\color{mygray2}\Huge $\mathbf Q_2$}}
\put(190,163){{\color{mygray2}\Huge $\mathbf Q_1$}}
\put(340,70){{\color{mygray2}\Huge $\mathbf Q_3$}}
  \put(0,110){
    \begin{minipage}[c]{\textwidth}  
\[
Q  = 
\begin{bmatrix}  \hspace{2mm} 
    \shrinkthis[1]{0.9}{2k_{p_1}}  &   -1     & 0  & \hspace{-9mm}\cdots\quad  0     & 
\jala 0 & \jala\cdots & \jala & \jala   \cdots  & \jala   0  \\ 
 -1  & \hspace{-1mm}\shrinkthis[1]{0.9}{2[k_{p_2}\hspace{-1mm}-k_{p_1}]}  &         &   
 & \jala -\mu_{\shrinkthis[0.6]{0.7}{2}_1}  & \jala \ddots & \jala  & \jala    & \jala\vdots   \\ 
  0   &    \qquad\ddots   & &    & 
\jala\vdots & \jala-\mu_{\shrinkthis[0.6]{0.7}{3}_2}   & \jala  & \jala\\ 
 \vdots   &  &  & -1 &\jala   &\jala  \vdots &\jala  &\jala  \ddots &\jala  0 \\[1.3mm]
 0   &     \hspace{-12mm} \cdots   & -1  &  \shrinkthis[1]{0.9}{2[k_{p_m}\hspace{-2mm}-\!k_{p_{m-1}}]}     &\jala   -\mu_{\shrinkthis[0.6]{0.7}{m}_1}       &\jala        -\mu_{\shrinkthis[0.6]{0.7}{m}_2}      &\jala       \cdots       &\jala     \shrinkthis[1]{0.9}{-\mu_{\shrinkthis[0.6]{0.7}{m}_{m-1}} } &\jala 0        \\[2mm]
0  &  -\mu_{\shrinkthis[0.6]{0.7}{2}_1}   & \cdots  & -\mu_{\shrinkthis[0.6]{0.7}{m}_1}  & 
\jala 2a_1  &   \jala \shrinkthis[1]{0.9}{b_2\mu_{\shrinkthis[0.6]{0.7}{2}_1} }    &   \jala  &  \jala & \hspace{-14mm} \cdots \qquad \shrinkthis[1]{0.9}{b_m\mu_{\shrinkthis[0.6]{0.7}{m}_1} }    \\
  \vdots & \hspace{-12mm}\ddots  &    & \vdots  &  \jala  \shrinkthis[1]{0.9}{b_2\mu_{\shrinkthis[0.6]{0.7}{2}_1} }  &   \hspace{-1.5mm} \dgdots{55}  &  \jala  &   \jala &   \jala  \vdots   \\ 
  &  &  \ \ddots  &  &   \jala \vdots &   \jala\shrinkthis[1]{0.9}{b_3\mu_{\shrinkthis[0.6]{0.7}{3}_2}}  &  \jala &   \jala  &                  \\ 
\vdots&&\ddots & -\mu_{\shrinkthis[0.6]{0.7}{2}_{m-1}} &    &\jala \vdots& \jala   & \jala   \ddots & \jala    \shrinkthis[1]{0.9}{b_m\mu_{\shrinkthis[0.6]{0.7}{m}_{m-1}} } \\[1.3mm]
  0    &  \hspace{-12mm} \cdots    &       \cdots        &    0 
   & \mbox{\hspace{0mm}     \shrinkthis[1]{0.9}{b_m\mu_{\shrinkthis[0.6]{0.7}{m}_1} }     }
 & \mbox{\hspace{-2mm}    \shrinkthis[1]{0.9}{b_m\mu_{\shrinkthis[0.6]{0.7}{m}_2}}    }
  &  \mbox{\hspace{-5mm}    $\cdots$     }
  & \mbox{\hspace{-2mm}  \shrinkthis[1]{0.9}{b_m\mu_{\shrinkthis[0.6]{0.7}{m}_{m-1}} }}
 & \mbox{\hspace{-3mm}  $2a_m$         }
 \end{bmatrix}  
\kronecker I
\]
    \end{minipage}
  }
\end{picture}

\noindent is positive semidefinite, \rref{lego} holds, 
\begin{subequations}\label{794}
  \begin{eqnarray}
  \label{794a} 
&\displaystyle k_{p_1} =  k_{d_1}& \\
  \label{794b} 
&\displaystyle \min\left\{a_1,k_{p_1}\right\} > \eta_2^2 + \eta_3^2
\end{eqnarray}
\end{subequations}
 and, for  all   $i\in [2,m]$, 
\begin{subequations}  \label{795}
  \begin{eqnarray}
  \label{795a}
&\displaystyle k_{p_i} = \frac{1}{2}\Big[k_{p_{i-1}} + \sqrt{k_{p_{i-1}}^2 +4 k_{d_i}}\ \Big] & \\
  \label{795b}
&\displaystyle k_{p_i} > \shrinkthis[1]{0.95}{[\eta_2^2+\eta_3^2]}\left(\mbox{\small $\displaystyle\prod_{j=0}^{i-1}$}\beta_j\right)^2 +\, k_{p_{i-1}} &
\\
&\displaystyle   \label{795c} a_{i} > \shrinkthis[1]{0.95}{[\eta_2^2+\eta_3^2]}\left(\mbox{\small $b_i\displaystyle\prod_{j=0}^{i-1}$}\beta_j\right)^2&
  \end{eqnarray}
\end{subequations}
where $\beta_0 := 1$, ---see \rref{650} and \rref{771} for other definitions.  Then, the origin is uniformly globally asymptotically stable.
\end{proposition}
\begin{remark}
  Note, from \rref{771}, that $\beta_j$ depends on $k_{d_j}$, $k_{p_j}$ and $b_j$ therefore, the right-hand side of \rref{795b} depends on the control gains indexed up to $i-1$ only. Thus, by a suitable definition of $k_{d_i}$ in \rref{795a} both conditions  \rref{795a} and \rref{795b} may be met simultaneously. 

\end{remark}
\noindent {\bf Proof of Proposition \ref{prop1}}. 
Equation \rref{795a} implies, in view of \rref{771.5}, that $\eta_{i_k} = 0$ for all applying $i$ and $k$ therefore, $Q= -\mathcal A^\top P - P \mathcal A$ with $P=I$ corresponds to the matrix $Q$ in Proposition \ref{prop1}. This matrix is positive semidefinite if (i) the shadowed block $Q_1\kronecker I$ is positive definite, which holds if $k_{p_i}>k_{p_{i-1}}$ and $k_{p_i}[k_{p_i} - k_{p_{i-1}}] > 1$ for all $i\leq m$; (ii) the shadowed block $Q_3\kronecker I$ is positive definite, which holds for sufficiently large values of $a_i > a_{i-1}$. Under similar conditions (strengthened if necessary) the Schur complement of $Q_3\kronecker I$ {\em i.e.}, $[Q_1 - Q_2^\top Q_3^{-1}Q_2]\kronecker I $, is positive. One may draw the same conclusion for $Q\,-$\,\textonehalf\,diag$\{Q\}$ by enforcing the lower bounds on $a_i$ and $k_{p_i}$ as in \rref{795}. Thus, consider the function defined in \rref{789} with $P= I$ then, using $Q\,-$\,\textonehalf\,diag$\{Q\}\geq 0$ we obtain, from \rref{790},
\begin{eqnarray}
  \dot W &\leq & - \sfrac{1}{2} \left(\sum_{i=1}^m [k_{p_i} - k_{p_{i-1}}]\norm{\topxii}^2 + a_i\norm{\vartheta_i}^2\right)
- \left(\sum_{i=1}^m \mbox{\small $\displaystyle\prod_{j=0}^{i-1}$}\beta_j\topxii^\top - \mbox{\small $\displaystyle b_{i}\prod_{j=0}^{i-1}$}\beta_j\vartheta_i^\top \right)\dot \opxi_1^*. 
\label{800}
\end{eqnarray}
Next, considering once more \rref{650}, we obtain
\[
\big(\topxii^\top+\vartheta_i^\top\big)\dot\opxi_1^* \leq \sfrac{1}{2}\sum_{i=1}^m\shrinkthis[1]{0.95}{[\eta_2^2+\eta_3^2]}\Big(\norm{\topxii}^2 + \norm{\vartheta_i}^2\Big) + \eta_1\Big(\norm{\topxii} + \norm{\vartheta_i}\Big) + \Big[\norm{\dtq}^2 + \norm{\vartheta_0}^2\Big]m
\]
hence in view of \rref{794} and \rref{795} we conclude that there exist positive constants $\alpha_{\shrinkthis[0.6]{0.7}{1}_i}$, $\alpha_{\shrinkthis[0.6]{0.7}{2}_i}$ for all $i\in\{1\ldots, m\}$, such that the total derivative of 
\[
\mathcal V(t,\opchi) =  \sfrac{1}{2}\left(
\dtq^\top D(\tq+\qd(t)) \dtq + k_{p_0} \norm{\tq}^2 + \frac{k_{d_0}}{b_0} \norm{\vartheta_0}^2 \right)
+ \sfrac{1}{2}\norm{\opchi}^2
\]
along the closed-loop trajectories of \rref{765}, \rref{770} satisfies
\begin{eqnarray}
    \dot {\mathcal V} &\leq&  - \sfrac{1}{2} \left(\sum_{i=1}^m \alpha_{\shrinkthis[0.6]{0.7}{1}_i}\norm{\topxii}^2 + \alpha_{\shrinkthis[0.6]{0.7}{2}_i}\norm{\vartheta_i}^2\right)
+ \eta_1\Big(\norm{\topxii} + \norm{\vartheta_i}\Big) 
\notag \\ && \quad 
 -\Big[\sfrac{a_0k_{d_0}}{b_0}-m\Big]\norm{\vartheta_0}^2 + \big[k_ck_\delta + \sfrac{m+2}{2}\big] \norm{\dtq}^2
\label{820}
\end{eqnarray}
 --\cf Ineq. \rref{700}. The rest of the proof follows as for Theorem \ref{thm:main2}. 
\qed 

We wrap up  the section with a statement for the simplified model of flexible-joint manipulators simultaneously and independently introduced in \cite{BURZAR} and \cite{SPONGASME}. This corollary is another significant contribution to the theory of robot control since we are not aware of any result establishing uniform {\em global} asymptotic stability under output feedback. It must be recalled that very few results exist on position feedback tracking control of flexible-joint manipulators, one of the first of this nature is \cite{NICTOM3}.\\[-3mm]

\begin{corollary}[Flexible-joint manipulators]
Consider the system \rref{595} and suppose that\\[-9mm]
\begin{itemize}\itemsep=-3pt
\item item (2) of Assumption \ref{ass2} holds;
\item the reference trajectories satisfy \rref{608};
\item there exists $k_v>0$ such that  $\norm{\frac{\partial \mbox{\small \rm g}}{\partial q}} \leq k_v$.
\end{itemize}
Consider the controller defined by
\begin{subequations}\label{830}
  \begin{eqnarray}
\label{830a}
 \tau & =& K(q_2-q_1) + JK^{-1}u, \\
&& \hspace{-13mm} \mbox{\rm Equations \rref{750} with $m=2$, }\\
\label{830c} \opxi_1^* & =&  -k_{p_0} \tilde q - k_{d_0} \vartheta_0 + D(q_1)\ddot q_{1d} + C(q_1,\dot q_{1d})\dot q_{1d} + \g(q_1) + Kq_1,\\ 
  \end{eqnarray}
\end{subequations}
and assume that the control gains satisfy \rref{794}, \rref{795} with $m=2$ and $k_{d_0} b_0/a_0 > 2k_ck_\delta$. Then, the origin of the closed-loop system is uniformly globally asymptotically stable.
\end{corollary}
\begin{proof}
We have $m=2$; the matrix $Q$ in Proposition \ref{prop1} becomes
\[
Q  = 
\begin{bmatrix}  \hspace{2mm} 
    \shrinkthis[1]{0.9}{2k_{p_1}}  &   -1     & 0  &  0 \\
 -1  & 
 \shrinkthis[1]{0.9}{2[k_{p_2}\hspace{-1mm}-k_{p_1}]}  & -\mu & 0\\
\phantom{-}0   & -\mu  & 2a_1 & -b_2\mu \\
\phantom{-}  0   &  \phantom{-}0 & \hspace{-1.3mm}-b_2\mu & \phantom{-} 2a_2 
 \end{bmatrix}  
\kronecker I 
\]
where $\mu = k_{d_1}[k_{p_1}+a_1]$ and, by assumption, $k_{d_1}=k_{p_1}$. We see that $Q\,-$\,\textonehalf\,diag$\{Q\}$ is positive semidefinite if so are 
\[
\begin{bmatrix}  \hspace{2mm} 
    \shrinkthis[1]{0.9}{k_{p_1}}  &   -1   \\
 -1  & 
 \shrinkthis[1]{0.9}{\mbox{\textonehalf}[k_{p_2}\hspace{-1mm}-k_{p_1}]}
\end{bmatrix}, \quad 
\begin{bmatrix}  \hspace{2mm} 
 \shrinkthis[1]{0.9}{\mbox{\textonehalf}[k_{p_2}\hspace{-1mm}-k_{p_1}]}  & -\mu\\
 -\mu  & \mbox{\textonehalf}a_1 
\end{bmatrix}, \quad 
\begin{bmatrix}
\mbox{\textonehalf} a_1 & -b_2\mu \\
-b_2\mu & \phantom{-} a_2 
\end{bmatrix}
\]
which holds true if all constants are positive, $k_{p_1} [k_{p_2}-k_{p_1}] \geq 2$, $a_1 [k_{p_2}-k_{p_1}] \geq 4\mu^2$ and $a_1a_2 \geq 2b_2^2\mu^2$. 

Furthermore, condition \rref{786} holds if 
\[
  \sfrac{1}{2}\mbox{\rm diag}\{Q\} > (\eta_2^2+\eta_3^2)\mbox{\rm diag}\big\{ \norm{ [ P\mathcal B]_i }^2\big\}
\]
where in this case, since $P=I$, is equivalent to 
\[
\begin{bmatrix}  \hspace{2mm} 
    \shrinkthis[1]{0.9}{k_{p_1}}   &   0     & 0  &  0 \\
0  & 
 \shrinkthis[1]{0.9}{[k_{p_2}\hspace{-1mm}-k_{p_1}]}   & 0 & 0\\
0   & 0  & a_1  & 0 \\
0  & 0 & 0 & a_2
 \end{bmatrix}  
 > 
\begin{bmatrix}  \hspace{2mm} 
 1  &   0     & 0  &  0 \\
0  & 
\beta^2   & 0 & 0\\
0   & 0  & b_1^2 & 0 \\
0  & 0 & 0 & b_1^2\beta^2 
 \end{bmatrix}\big[\eta_2^2+\eta_3^2]  
\]
where $\beta= k_{p_1} + b_1k_{d_1}$ or, since $k_{p_1}=k_{d_1}$, $\beta= k_{p_1}(1 + b_1)$. 

All these conditions, which correspond to \rref{794} and \rref{795}, may be met for appropriate values of $k_{p_2}>k_{p_1}$ and $a_2>a_1$.  
\end{proof}
\section{Discussion and open problems}
\label{sec:disc}
There is an interesting passivity interpretation to the control approach presented in the previous sections. For simplicity, the following discussion focuses on the system of relative degree 2. 

The output-feedback tracking controller \rref{bla} may be regarded as composed of two parts, a set-point control law of Proportional-Derivative with gravity cancellation (similar to that from \cite{KEL1}) and a second part playing the role of a  ``feedforward'', this is
\begin{equation}
  \label{eq:200}
v :=  D(q)\ddqd + C(q,\dqd)\dqd + k_pq_d.
\end{equation}
Then, in view of forward completeness, it follows from Assumption 1 and expression \rref{bndonqd} that $v(t) = D(q(t))\ddqd(t) + C(q(t),\dqd(t))\dqd(t) + k_p\qd(t)$  is uniformly bounded. With this notation the control input \rref{90c} may be re-written as
\begin{equation}
\label{eq:250}
  u = -k_p q - k_d\vartheta + g(q) + v
\end{equation}
and the closed-loop system takes the form 
\begin{equation}
  \label{eq:300}
  \dot x = F(x) + G(x)w, \quad x:= [q^\top\, \dot q^\top\, \vartheta^\top]^\top
\end{equation}
where 
\[  F:=
  \begin{bmatrix}
    \dot q \\
D(q)^{-1} [-C(q,\dot q)\dot q -k_p q - k_d\vartheta]\\
-a\vartheta + b\dot q
  \end{bmatrix}, \quad w:=
  \begin{bmatrix}
    0\\ v \\ \dqd
  \end{bmatrix},
\quad G(x):=
  \begin{bmatrix}
    0 & 0 & 0 \\ 0 & D(q)^{-1} & 0 \\ 0 & 0 & -b
  \end{bmatrix}.
\]
The dynamics $\dot x = F(x)$ corresponds to that of the closed-loop system with the set-point controller from \cite{KEL1} taking as set-point reference $\qd =0$ that is, \rref{eq:250} with $v = 0$. Now, the total time-derivative of the {\em storage} function
\begin{equation}
  \label{eq:400}
  V(q,\dot q,\vartheta) = \frac{1}{2}\left( \dot q^\top D(q) \dot q + k_p \norm{q}^2 + \frac{k_d}{b}\norm{\vartheta}^2\right)
\end{equation}
along the trajectories of \rref{eq:300}, takes the form 
\begin{equation*}
\dot V = -\frac{k_da}{b}\norm{\vartheta}^2 - \vartheta^\top b \dqd + \dot q^\top v.
\end{equation*}
Since $b \dqd$ is bounded, for any given positive number $\lambda$ we have
\begin{equation}\label{above}
  \norm{\vartheta}\geq \lambda \ \Longrightarrow \ \dot V \leq -\left(\frac{k_da}{b} - \frac{bk_\delta^2}{\lambda}\right)\norm{\vartheta}^2  + \dot q^\top v
\end{equation}
which implies that the map $v\mapsto \dot q$ is passive for ``large'' values of the filter output $\vartheta$. It is also to be noted that the filter \rref{90a}, \rref{90b} defines an output strictly passive map $\dot{\tilde q}\mapsto \vartheta$ with finite $\mathcal L_2$ gain. This provides an input-output-based rationale to establish uniform global boundedness: if the outputs of the filter \rref{filter} grow ``large'' the feedback-interconnected  system ``becomes'' passive hence, input-output stable.

A key feature of the controller \rref{90} and which is exploited to extend the results to systems of higher relative degree is that it guarantees uniform global asymptotic stability. This is remarkable since, according to Malkin, this implies {\em total stability}\footnote{Concept introduced in \cite{MALKIN54} and known in modern literature as local Input-to-State-Stability.} that is, robustness with respect to bounded disturbances. However, this property is guaranteed only locally; establishing {\em global} Input-to-State-Stability via a (strict) Lyapunov function remains an open challenge both, for Lagrangian systems under output-feedback and more generally, for nonlinear-time varying systems. 

In the context of the ``global tracking problem'' for robot manipulators, a remarkable paper pursuing this direction is \cite{NUNHSU} however, the main result in this reference relies on the instrumental assumption that the system is naturally damped by viscous friction forces --see the model \rref{confriccion}. Then, a direct computation shows that by applying 
\[  u = -k_p q - k_d\vartheta + g(q) + F\dot q_d + v
\]
to system \rref{confriccion} we obtain the ``power'' balance equation 
\begin{equation}
 \label{eq:410}
 \dot V \leq  -\frac{k_da}{b}\norm{\vartheta}^2 - ( F - k_ck_\delta)\norm{\dtq}^2 + [0\,\ \dot{\tilde q}^\top\, \vartheta^\top]^\top w
\end{equation}
which is in clear contrast with \rref{140:b}. The closed-loop system now defines an output strictly passive map $w\mapsto [\vartheta,\,\dtq]$. Furthermore, either for sufficiently large $F$ or for sufficiently small $k_\delta$ (which tantamounts to imposing ``slow'' reference trajectories) we see that $w=0$ implies $\dot V \leq 0$. 

Since $\dot V$ is only negative semidefinite for the system without input, it does not qualify as an input-to-state-stable Lyapunov function. Nonetheless, the authors of \cite{NUNHSU} smartly establish input-to-state-stability under output feedback. It is worth emphasizing that $V$ constitutes a so-called Lyapunov function ``{\em satisfying Lasalle's conditions}''; conditions to establish input-to-state-Stability in such context have been studied with certain degree of generality, for instance in \cite{ANGPD}, where the main result is also motivated by a robot control problem,  as well as in \cite{MAZBOOK} and a number of references therein. 

Roughly speaking, in \cite{ANGPD} it is established that a {\em time-invariant} system 
\begin{equation}
  \label{eq:600}
  \dot x = f(x,w) \qquad x:=[x_1^\top\, x_2^\top]^\top \in \mR^m
\end{equation}
with input $w$, admits an  input-to-state-stability Lyapunov function $V$ (positive definite and proper) such that 
\[
 \frac{\partial V}{\partial x}f(x,w) \leq -\alpha_1(\norm{x_1}) + \gamma(\norm{w}),\quad \alpha_1\in\mathcal K_\infty, \ \gamma\in \mathcal K
\]
provided that:
\begin{itemize}
\item  there exist positive-definite proper functions $V_1$ and $V_2$;
\item  there exist class $\mathcal K$ functions $\alpha_{11}$, $\alpha_{12}$, $\alpha_{21}$, $\alpha_{22}$, $\gamma_1$, $\gamma_2$ such that
\begin{equation}
  \label{eq:550}
  \frac{\partial V_1}{\partial x}f(x,w) \leq -\alpha_{11}(\norm{x_2}) + \alpha_{12}(\norm{x_2})\gamma_1(\norm{w}) 
\end{equation}
\begin{equation}
  \label{eq:560}
 \frac{\partial V_2}{\partial x}f(x,w) \leq -\alpha_{21}(\norm{x}) + \alpha_{22}(\norm{x_2})\gamma_2(\norm{w});  
\end{equation}
\item the functions $\alpha_{22}$ and $\alpha_{11}$ have the same order of growth. 
\end{itemize}
The first property {\em i.e.}, the existence of $V_1$ satisfying \rref{eq:550}, in \cite{ANGPD} is called quasi-Input-to-State-Stability. The prefix ``quasi'' is motivated by the fact that $V_1$ satisfies ``Lasalle''-type conditions for global asymptotic stability, when $w \equiv 0$. The second property is  referred to as Input-Output-to-State stability with output $x_2$ and it is a notion of detectability for nonlinear systems. Now, note that \rref{eq:410} is of the form \rref{eq:550} however, this is not the case for \rref{above} which fails to satisfy this condition since the arguments of $\alpha_{11}$ and $\alpha_{12}$ are different. The property may be established if we assume $F=F^\top>0$.

 To the best of our knowledge, constructing a strict Lyapunov function for the problem stated in Definition \ref{def:pblm} is an open problem which is illustrated by but not limited to the case of the controller \rref{90}. In a general nonlinear context, the state of the art in constructing Lyapunov functions for nonlinear time-varying systems relies on Lyapunov functions that have negative semi-definite derivatives --see \cite{MAZBOOK}, as opposed to $V$ defined in \rref{eq:400}. The construction of an Input-to-State-Stability Lyapunov function $V$ for systems satisfying 
\[
\norm{x_1}\geq \lambda\ \Longrightarrow \  \frac{\partial V}{\partial x}f(x,w) \leq -\alpha_1(\norm{x_1}) + \alpha_2(\norm{x_2})\gamma(\norm{w})
\]
rather than \rref{eq:550} is, in our opinion, another challenging and interesting open problem. 

Last but not least, we remark that keeping in mind \cite{MAZPRADAY}, the key property that allows for the result in Theorem 4 is the skew-symmetry of $\dot{\overparen{D(q)}} - 2C(q,\dot q)$. Therefore, a {\em geometric} interpretation of this property is fundamental to establish a statement that apply to a larger class of Euler-Lagrange systems, if such extension is possible at all. 

\section{Conclusions}
\label{sec:concl}
A dynamic position-feedback controller for Lagrangian systems without dissipative forces and a constructive proof of uniform global asymptotic stability for the closed-loop system is presented. Our simplest result, which closes a significant chapter on output feedback control of nonlinear systems,  implicitly establishes, {\em without a Lyapunov function}, the very intuitive conjecture that the damping necessary to stabilize the system may be introduced through a simple approximate-derivatives filter. Instrumental to the proof is that such filter has finite DC gain. Furthermore, it is proved that such a naive control design may be applied with success to systems of higher relative degree, using cascaded approximate differentiators. 

The importance of these results can hardly be overestimated; we believe that our findings may pave the way towards a simple observer-less dynamic output feedback control approach inspired by the backstepping method but avoiding the cumbersome highly nonlinear resulting control laws. On the grounds of systems analysis, we have briefly sketched new challenging open problems on construction of strict Lyapunov functions for systems satisfying Lasalle's conditions modulo the gain of a passive filter. Research in these directions is currently pursued.

\vskip 9pt
\centerline{\bf Acknowledgements}

The author is indebted to R. Ortega for the arbitrarily large (but bounded) number of discussions on this problem since he proposed it to the author as PhD subject, 20 years ago. The construction of the proofs reflects the little the author has learned out of the much E. Panteley has taught him for more than 17 years.

\clearpage 

{\small
\def\loria{Lor\'{\i}a}
  \def\nesic{Ne\v{s}i\'{c}\,}\def\astrom{{\SortNoop{As}\AA}str{\"{o}}m\,}\def\%
nonumero{\def\numerodeitem{}}

}
\appendix

    \setlength{\baselineskip}{22pt plus 2pt minus 1pt}
\section{Appendix}

Without loss of generality, we fix $m=4$. 

\begin{description}
 \item[\bf Case $i=1$:] 
   \begin{minipage}[t]{5cm}
     \begin{eqnarray*}
           \opxi_1^* &=& \opxi_1^*\\
           \dtopxi{1} &=& -k_{p_1}\topxi{1} + k_{d_1}\vartheta_1 + \topxi{2} - \dot\opxi_1^*\\
           \dot\vartheta_1 & = & -a_1\vartheta_1 + b_1\dot\opxi_1^* - k_{d_1}\topxi{1}
     \end{eqnarray*}
   \end{minipage}
%
 \item[\bf Case $i=2$:] 
     \begin{eqnarray*}
           \dot \opxi_2^* &=& -k_{p_1}\dtopxi{1} + k_{d_1}\dot \vartheta_1\\
                          &=& -k_{p_1}\big[ -k_{p_1}\topxi{1} + k_{d_1}\vartheta_1 + \topxi{2} - \dot\opxi_1^* \big] + k_{d_1}\big[ -a_1\vartheta_1 + b_1\dot\opxi_1^* - k_{d_1}\topxi{1} \big]\\
                         &=& \big[ k_{p_1}^2 - k_{d_1}^2\big]\topxi{1}  - k_{d_1}\big[ k_{p_1} + a_1\big] \vartheta_1-k_{p_1}\topxi{2}  + \big[k_{p_1} + k_{d_1}b_1\big]\dopxi{1}^* \\
           \dtopxi{2} &=& -\big[ k_{p_2} - k_{p_1}\big] \topxi{2}+ k_{d_2}\vartheta_2 + \topxi{3} - \big[ k_{p_1}^2 - k_{d_1}^2\big]\topxi{1} +k_{d_1}\big[ k_{p_1} + a_1\big] \vartheta_1-\big[k_{p_1} + k_{d_1}b_1\big]\dopxi{1}^* 
\\ 
\dtopxi{2} &=& -\big[ k_{p_2} - k_{p_1}\big] \topxi{2}+ k_{d_2}\vartheta_2 + \topxi{3} - \eta_{2_1}\topxi{1} + \mu_{2_1}\vartheta_1 - \beta_1\dopxi{1}^* 
\\ 
\dot\vartheta_2 & = & -a_2\vartheta_2 + b_2\Big[ 
\big[ k_{p_1}^2 - k_{d_1}^2\big]\topxi{1} - k_{d_1}\big[ k_{p_1} + a_1\big] \vartheta_1 -k_{p_1}\topxi{2} +  \big[k_{p_1} + k_{d_1}b_1\big]\dopxi{1}^*
\Big]- \big[k_{d_2}-\sigma_2\big]\topxi{2}
\\ & = & 
-a_2\vartheta_2 + b_2\eta_{2_1}\topxi{1}-b_2\mu_{2_1}\vartheta_1+ b_2\beta_1\dopxi{1}^* -  k_{d_2} \topxi{2}
     \end{eqnarray*}
%
 \item[\bf Case $i=3$:] 
     \begin{eqnarray*}
           \dot \opxi_3^* &=& -k_{p_2}\dtopxi{2} + k_{d_2}\dot \vartheta_2 \\
\dtopxi{3} &=& -k_{p_3}\topxi{3} + k_{d_3} \vartheta_3 + \topxi{4} + 
k_{p_2} \Big[-\big[ k_{p_2} - k_{p_1}\big] \topxi{2}+ k_{d_2}\vartheta_2 + \topxi{3} - \big[ k_{p_1}^2 - k_{d_1}^2\big]\topxi{1}
\\   &&\qquad + k_{d_1}\big[ k_{p_1} + a_1\big] \vartheta_1-\big[k_{p_1} + k_{d_1}b_1\big]\dopxi{1}^* \Big]
-k_{d_2} \Big[
-a_2\vartheta_2 + b_2\Big[ 
\big[ k_{p_1}^2 - k_{d_1}^2\big]\topxi{1} \\
&& \qquad \quad - k_{d_1}\big[ k_{p_1} + a_1\big] \vartheta_1 -k_{p_1}\topxi{2} +  \big[k_{p_1} + k_{d_1}b_1\big]\dopxi{1}^*
\Big]- \big[k_{d_2}-\sigma_2\big]\topxi{2}   \Big]
\\
\dtopxi{3} &=& -\big[ k_{p_3} - k_{p_2}\big]\topxi{3} +  k_{d_3}\vartheta_3+ \topxi{4} -\big[ k^2_{p_2} - k_{d_2}^2\big] \topxi{2} + k_{d_1}\big[ k_{p_1} + a_1\big] \big[ k_{p_2} + k_{d_2}b_2\big]\vartheta_1\\
&&\quad  + k_{d_2}\big[ k_{p_2} + a_2\big] \vartheta_2- \big[ k_{p_1}^2 - k_{d_1}^2\big]\big[k_{p_2} + k_{d_2}b_2\big]\topxi{1}-\big[k_{p_2} + k_{d_2}b_2\big]\big[k_{p_1} + k_{d_1}b_1\big]\dopxi1^* \\ &&\qquad + \big[k_{p_1}k_{p_2}+ k_{p_1}b_2k_{d_2}-\sigma_2k_{d_2}\big]\topxi2
\\
\dtopxi{3} &=&-\big[ k_{p_3} - k_{p_2}\big]\topxi{3} +  k_{d_3}\vartheta_3+ \topxi{4} - \eta_{3_2}\topxi{2} + \mu_{3_1}\vartheta_1+ \mu_{3_2}\vartheta_2-\eta_{3_1}\topxi{1}-\beta_2\beta_1\dopxi{1}^*
 \\ 
\dot\vartheta_3 & = & -a_3\vartheta_3 + b_3\Big[\beta_2\beta_1\dopxi{1}^* + \eta_{3_1}\topxi{1} -k_{p_2}\topxi{3}-\mu_{3_1}\vartheta_1 + \eta_{3_2}\topxi{2}-\mu_{3_2}\vartheta_2\Big] -\big[k_{d_3}-\sigma_3\big]\topxi{3} \\ 
\dot\vartheta_3 & = & -a_3\vartheta_3 + b_3\Big[\beta_2\beta_1\dopxi{1}^* + \eta_{3_1}\topxi{1} -\mu_{3_1}\vartheta_1 + \eta_{3_2}\topxi{2}-\mu_{3_2}\vartheta_2\Big] -k_{d_3}\topxi{3}
   \end{eqnarray*}
 \item[\bf Case $i=4$:] 
     \begin{eqnarray*}
           \dot \opxi_4^* &=& -k_{p_3}\dtopxi{3} + k_{d_3}\dot \vartheta_3 
\\ \dtopxi{4} &=&            - k_{p_4}\topxi{4} + k_{d_4}\vartheta_4 + k_{p_3}\Big[ -\big[ k_{p_3} - k_{p_2}\big]\topxi{3}+ k_{d_3}\vartheta_3 + \topxi{4} - \big[ k_{p_2}^2 - k_{d_2}^2-k_{p_1}k_{p_2}\big]\topxi{2} 
\\ &&\quad 
+ k_{d_1}\big[ k_{p_1} + a_1\big] \big[ k_{p_2} + k_{d_2}b_2\big]\vartheta_1+ k_{d_2}\big[ k_{p_2} + a_2\big]\vartheta_2 - \big[ k_{p_1}^2 - k_{d_1}^2\big]\big[ k_{p_2} + k_{d_2}b_2\big]\topxi{1}
\\ &&\quad
-\big[ k_{p_2} + k_{d_2}b_2\big]\big[ k_{p_1} + k_{d_1}b_1\big]\dopxi{1}^*\Big]
-k_{d_3}\Big[ -a_3\vartheta_3 + b_3\Big(\big[ k_{p_2} + k_{d_2}b_2\big]\big[ k_{p_1} + k_{d_1}b_1\big]\dopxi{1}^*
\\ &&\quad 
+\big[k_{p_2} + k_{d_2}b_2\big]\big[ k_{p_1}^2 - k_{d_1}^2\big]\topxi{1}- k_{p_2}\topxi{3} 
- \big[k_{p_2} + k_{d_2}b_2\big]k_{d_1}\big[ k_{p_1} + a_1\big]\vartheta_1-
 \\ &&\quad 
\big[ k_{p_2}^2 - k_{d_2}^2-k_{p_1}k_{p_2}\big]\topxi{2}-k_{d_2}\big[ k_{p_2}+a_2\big]\vartheta_2 \Big)- \big[k_{d_3}-\sigma_3\big]\topxi{3}  
 \Big] \\
 \dtopxi{4} &=& -\big[ k_{p_4} - k_{p_3} \big]\topxi{4}+ k_{d_4}\vartheta_4 + k_{d_3}\big[k_{p_3} + a_3\big]\vartheta_3 +\big[ k_{p_2}+a_2\big]k_{d_2}\big[k_{p_3} + k_{d_3}b_3\big]\vartheta_2\\ &&\quad 
+k_{d_1}\big[ k_{p_1} + a_1\big]\big[k_{p_2} + k_{d_2}b_2\big]\big[k_{p_3} + k_{d_3}b_3\big]\vartheta_1
-\big[ k_{p_3}^2 - k_{d_3}^2\big]\topxi{3}\\ &&\quad -\big[ k_{p_2}^2 - k_{d_2}^2-k_{p_1}k_{p_2}\big]\big[ k_{p_3} + k_{d_3}b_3\big]\topxi{2} -\big[ k_{p_1}^2 - k_{d_1}^2\big]\big[ k_{p_2} + k_{d_2}b_2\big]\big[ k_{p_3} + k_{d_3}b_3\big]\topxi{1} \\ &&\quad -\big[ k_{p_1} + k_{d_1}b_1\big]\big[ k_{p_2} + k_{d_2}b_2\big]\big[ k_{p_3} + k_{d_3}b_3\big]\dopxi1^*+\big[k_{p_2}k_{p_3}+ k_{p_2}b_3k_{d_3}-\sigma_3k_{d_3}\big]\topxi{3}
\\
 \dtopxi{4} &=&-\big[ k_{p_4} - k_{p_3} \big]\topxi{4}+ k_{d_4}\vartheta_4 + \mu_{43}\vartheta_3 + \mu_{42}\vartheta_2 + \mu_{41}\vartheta_1 -\eta_{43} \topxi{3} -\eta_{42} \topxi{2} -\eta_{41} \topxi{1} -\beta_3\beta_2\beta_1\dopxi{1}^*
 \\ 
\dot\vartheta_4 & = & -a_4\vartheta_4 + b_4\Big[\prod_{j=1}^i\beta_j\dopxi{1}^* + \eta_{4_1}\topxi{1} -k_{p_3}\topxi{4}-\mu_{41}\vartheta_1 + \eta_{42}\topxi{2}-\mu_{42}\vartheta_2+ \eta_{43}\topxi{3}-\mu_{43}\vartheta_3\Big] \\ && \quad-\big[k_{d_4}-\sigma_4\big]\topxi{4}
\\ 
\dot\vartheta_4 & = & -a_4\vartheta_4 + b_4\Big[\prod_{j=1}^i\beta_j\dopxi{1}^* + \eta_{4_1}\topxi{1} -\mu_{41}\vartheta_1 + \eta_{42}\topxi{2}-\mu_{42}\vartheta_2+ \eta_{43}\topxi{3}-\mu_{43}\vartheta_3\Big] -k_{d_4}\topxi{4}
    \end{eqnarray*}
\end{description}
\end{document}